\documentclass[]{article}
\pdfoutput=1

\usepackage{amsthm}
\usepackage{amsmath}
\usepackage{amssymb}
\usepackage{graphicx}
\usepackage{subfigure}
\usepackage{enumerate}
\usepackage{url}
\usepackage{tikz}
\usetikzlibrary{arrows}

\newcommand{\RR}{\mathbb{R}}
\newcommand{\SYSTEMWR}{(\ref{eq:general_system})-(\ref{eq:reset_condition})}
\newcommand{\SYSTEMWRLinear}{(\ref{eq:general_system})-(\ref{eq:reset_condition})}
\newcommand{\s}{\mathfrak{s}}
\newcommand{\ts}{\tilde{\s}}
\newcommand{\tS}{\tilde{S}}
\newcommand{\tp}{\tilde{\phi}}
\newcommand{\tpsi}{\tilde{\psi}}
\newcommand{\T}{\mathbb{T}}
\newcommand{\Tt}{\T_T}
\newcommand{\tz}{\tilde{z}}

\newcommand{\bx}{\bar{x}}
\newcommand{\R}{{\mathcal{R}}}
\newcommand{\LL}{{\mathcal{L}}}
\newcommand{\tx}{\tilde{x}}

\newcommand{\conds}{H.1--H.2}
\newtheorem{lem}{Lemma}[section]
\newtheorem{remark}{Remark}[section]
\newtheorem{prop}{Proposition}[section]
\newtheorem{definition}{Definition}[section]
\newtheorem{theorem}{Theorem}[section]
\numberwithin{equation}{section}

\graphicspath{
{./}
}
\begin{document}
\title{Border collision bifurcations of stroboscopic
maps in periodically driven spiking models\thanks{This work has been financially
supported by the Large Scale Initiative Action REGATE,  by the Spanish
MINECO-FEDER Grants MTM2009-06973, MTM2012-31714 and the Catalan Grant
2009SGR859}}

\author{A. Granados \and M. Krupa \and F. Cl\'ement}
\date{}
\maketitle
\begin{abstract}
In this work we consider a general non-autonomous hybrid system based on the
integrate-and-fire model, widely used as simplified version of neuronal models
and other types of excitable systems. Our unique assumption is that the system
is monotonic, possesses an attracting subthreshold equilibrium point and is
forced by means of periodic pulsatile (square wave) function.\\
In contrast to classical methods, in our approach we use the stroboscopic map
(time-$T$ return map) instead of the so-called firing-map. It becomes a
discontinuous map potentially defined in an infinite number of partitions. By
applying theory for piecewise-smooth systems, we avoid relying on particular
computations and we develop a novel approach that can be easily extended to
systems with other topologies (expansive dynamics) and higher dimensions.\\
More precisely, we rigorously study the bifurcation structure in the
two-dimensional parameter space formed by the amplitude and the duty cycle of
the pulse. We show that it is covered by regions of existence of periodic
orbits given by period adding structures. They do not only completely describe
all the possible spiking asymptotic dynamics but also the behavior of the firing
rate, which is a devil's staircase as a function of the parameters.
\end{abstract}
\section{Introduction}
In the context of neuronal modeling (and similarly for hormone release), one
relies on systems that are able to exhibit large-amplitude responses under
certain stimuli.  An example of such behavior is an \emph{action potential} (AP)
or \emph{spike} of a neuron.
There exist accurate but quite complex models of APs, e.g. the Hudgkin-Huxley
equations~\cite{HodHux52}. Simpler models are the Morris-Lecar~\cite{MorLec81}
or the FitzHugh and Nagumo~\cite{FitHug61,NagAriYos62}. All these models use the
principle of excitability, and often exhibit slow/fast dynamics.\\
Hybrid systems with resets are used as approximate models of excitability, with
the discontinuity, that is the reset, used to mimic the spike.  In many contexts
a spike occurs when the system receives a sufficient amount of stimulus.\\
One can distinguish between different types of systems with resets. The simplest
are \emph{integrate-and-fire} models, for which the differential equation
exhibits very simple dynamics (it is constant or linear). Other models include
more complicated subthreshold dynamics including instabilities (which introduce
expanding dynamics) in order to provide more realistic representations of the
underlying phenomena. This is the case of the so-called Izhikevich
model~\cite{Izh03}. Other approaches~\cite{BreGer05,TouBre09,MenHugRin12} let
the threshold depend on the state variable (thus increasing the dimension of the
system) in order to obtain a more accurate model with possibly more complex
properties, as is the case of type III neurons~\cite{ClaPayFor08,Izh07}.

Due to the discontinuity in the trajectories given by the reset, one cannot
directly apply classical theory for smooth systems in order to study these type
of excitable systems. A typical approach to overcome this problem is to consider
the so-called \emph{adaption}, \emph{firing} or \emph{impact} map, which is a
Poincar\'e map onto the
threshold~\cite{KeeHopRin81,CooBre99,CooThuWed12,TouBre09,JimMihBroNieRub13,DipKruTorGut12};
that is, one takes an initial condition on the threshold (which leads to a
reset) and one then integrates the system until the threshold is reached again.
This is a well known technique also in other disciplines involving
piecewise-smooth dynamics (see~\cite{BerBudChaKow08} for examples); this map
becomes the composition of smooth maps and hence provides a regular version of
the system which can be studied by means of classical tools for smooth
systems.\\
Note that, in order to compute the firing map, one needs to know the time needed
by the system to perform the next spike (firing times). Even assuming linearity,
these times need to be computed numerically, as they involve transcendental
equations.  Hence, one also relies on numerical computations in order to derive
properties of the firing map.

When introducing a $T$-periodic forcing to the system (a periodic current in the
neuronal context), the firing map must include time as a variable. In fact, for
the one-dimensional case, it becomes a map that returns the next firing time.
Hence, in order to study the existence of periodic orbits, explicit knowledge of
the behaviour of these times becomes crucial, as one needs to check whether the
difference between consecutive spiking times is congruent (rational multiple)
with $T$. Results in the context of firing maps have been obtained in specific
cases. For example, it was shown in~\cite{KeeHopRin81} that, for a linear
system, by using the explicit expression for the flow, one can approximate the
firing map by a certain a circle map and derive information about rotation and
firing numbers of existing periodic orbits. By combining this same technique
with numerical simulations, it was shown in~\cite{CooOweSmi01} that one obtains
similar results when introducing a periodic forcing to a two-dimensional
integrate-and-fire-or-bursting system. There authors studied  in more detail the
existence of periodic orbits (mode locking) and showed numerically that the
firing number may follow a devil's staircase as a function of parameters.

In this work we present a new approach to study periodically forced systems with
resets based on the use of the so-called stroboscopic map (time-$T$ return map).
One of the advantages of our approach is that it leads to a general setting for
periodically forced systems.  Already in the smooth case such systems are better
understood by means of the stroboscopic map than with Poincar\'e maps onto
transverse sections in the state space. In the latter, as mentioned above, one
has to check for congruency between the passage times and the period of the
forcing, which requires explicit knowledge of the solutions of the system or
approximations by circle maps in order to study the existence of periodic
orbits. However, when dealing the with stroboscopic map, this is reduced to the
study of existence of fixed (or periodic) points of this map, which comes by
applications of classical results as the implicit function theorem (when
studying perturbations) or Brouwer's theorem.\\
The use of the stroboscopic maps has been avoided in systems with resets because
it becomes a discontinuous (piecewise-smooth) map. The sets of initial condition
for which this map becomes smooth are identified by the number of spikes
exhibited by their trajectories when flowed for a time $T$, and hence it is
potentially defined in an arbitrarily large number of partitions. However, in
this work we use recent results in non-smooth systems to face these
discontinuities and rigorously study a general and large class of hybrid systems
with resets.\\
More precisely, we consider a one-dimensional monotonic system with an
attracting subthreshold equilibrium point under a $T$-periodic forcing
consisting of a periodic square wave function. This is given by a pulse of
amplitude $A$ and duration $dT$, with $0\le d\le1$ (duty cycle), and is $0$ for
the rest of the period $T$. Such a periodic stimulus is used in  many contexts,
as the stimulation of an excitable cell (e.g., an hormone release) frequently
occurs in a pulsatile form. Such a periodic forcing adds indeed an extra
non-smoothness, as one switches from one autonomous system to another one at the
switching times, when the pulse is switched on or off.  However, this
singularity consists of a discontinuity on the field that does not add nor
remove extra dynamical objects nor bifurcations.  As usual when dealing with
piecewise-smooth fields, in this case the solutions of the system are continuous
and obtained by proper concatenation of the solutions of the both fields, as
long as the threshold is not reached.\\
Our results consist of a rigorous and detailed description of the bifurcation
scenario formed by the parameters describing the input pulse: its amplitude $A$
and its duty cycle $d$. We show that there exists an infinite number of regions
in this two-dimensional parameter space for which the stroboscopic map exhibits
a fixed point. The regions in between are covered by regions of existence
periodic orbits following the so-called period adding structure. The rotation
number associated with these periodic orbits is a devil's staircase as a
function of the parameters.  This allows us to show that the asymptotic firing
rate of the system also follows a devil's staircase when the mentioned
parameters are varied.\\

\noindent This work is organized as follows.\\
In section~\ref{sec:results} we describe the system, announce our results and
describe the state of the art in piecewise-smooth maps relevant for our work.
In section~\ref{sec:proof_of_results} we prove our results, and in
section~\ref{sec:examples} we validate them by providing numerical computations
for three examples.

\section{Background and results}\label{sec:results}
Let us consider a non-autonomous periodic system given by
\begin{equation}
\dot{x}=f(x)+I(t),\,x\in\RR
\label{eq:general_system}
\end{equation}
with $f(x)\in C^\infty(\RR)$ and $I(t)$ a $T$-periodic square-wave function
\begin{equation}
I(t)=\left\{
\begin{aligned}
&A&&\text{if }t\in\left(nT,nT+dT\right]\\
&0&&\text{if }t\in(nT+dT,(n+1)T],\\
\end{aligned}\right.
\label{eq:forcing}
\end{equation}
with $0\le d\le 1$. Let us submit system~\eqref{eq:general_system} to the reset
condition
\begin{equation}
x=\theta \longrightarrow x=0;
\label{eq:reset_condition}
\end{equation}
that is, the trajectories of system~\eqref{eq:general_system} are
instantaneously reset to $0$ whenever they reach the threshold given by
$x=\theta$. This provides a new system which is in the class of \emph{hybrid}
systems, as it combines an algebraic condition with a differential equation. Due
to the generality of the function $f$, this system represents a large class of the
\emph{spiking} models which are used as simplification of slow-fast systems
modeling excitable cells as neurons. Hence, we will refer to the discontinuities
given by the reset condition as spikes.\\

Let us assume that the system
\begin{equation}
\dot{x}=f(x)
\label{eq:autonomous_system}
\end{equation}
satisfies the following conditions
\begin{enumerate}[H.1]
\item it possesses an attracting equilibrium point
\begin{equation}
0<\bx<\theta,
\label{eq:critical_point}
\end{equation}
\item $f(x)$ is monotonic decreasing function in $[0,\theta]$:
\begin{equation*}
f'(x)<0,\;0\le x\le \theta.
\end{equation*}
\end{enumerate}

\unitlength=\textwidth
\begin{figure}
\begin{center}
\begin{picture}(1,0.5)
\put(0,0.35){
\subfigure[\label{fig:d-invA_generic_regions}]
{\includegraphics[angle=-90,width=0.5\textwidth]{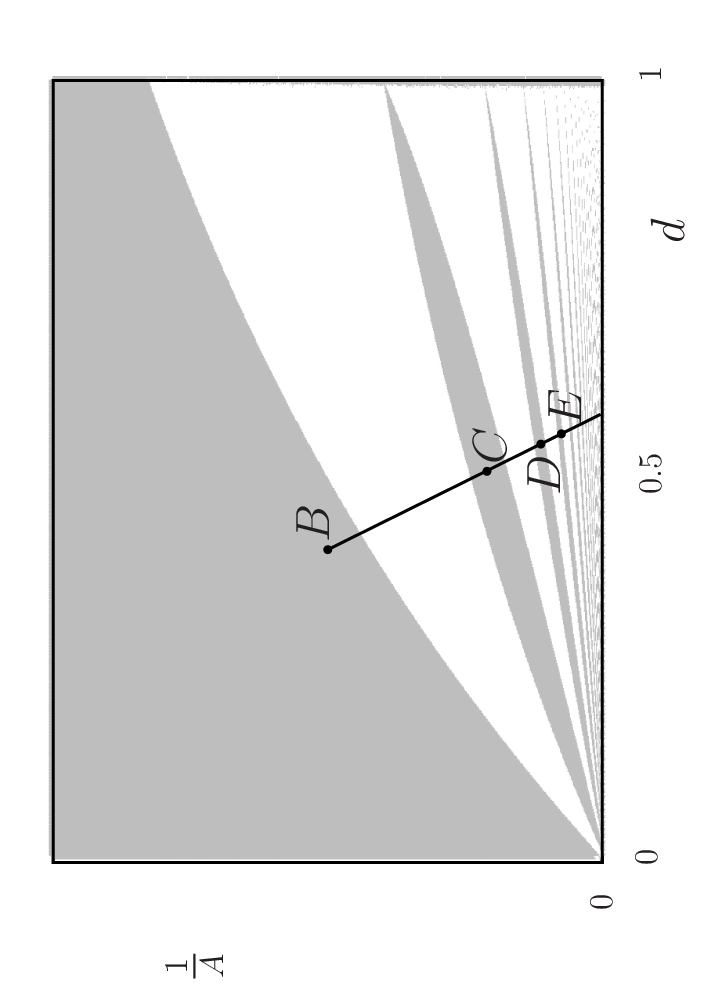}}
}
\put(0.5,0.35){
\subfigure[\label{fig:d-invA_generic_1dscann}]
{\includegraphics[angle=-90,width=0.5\textwidth]{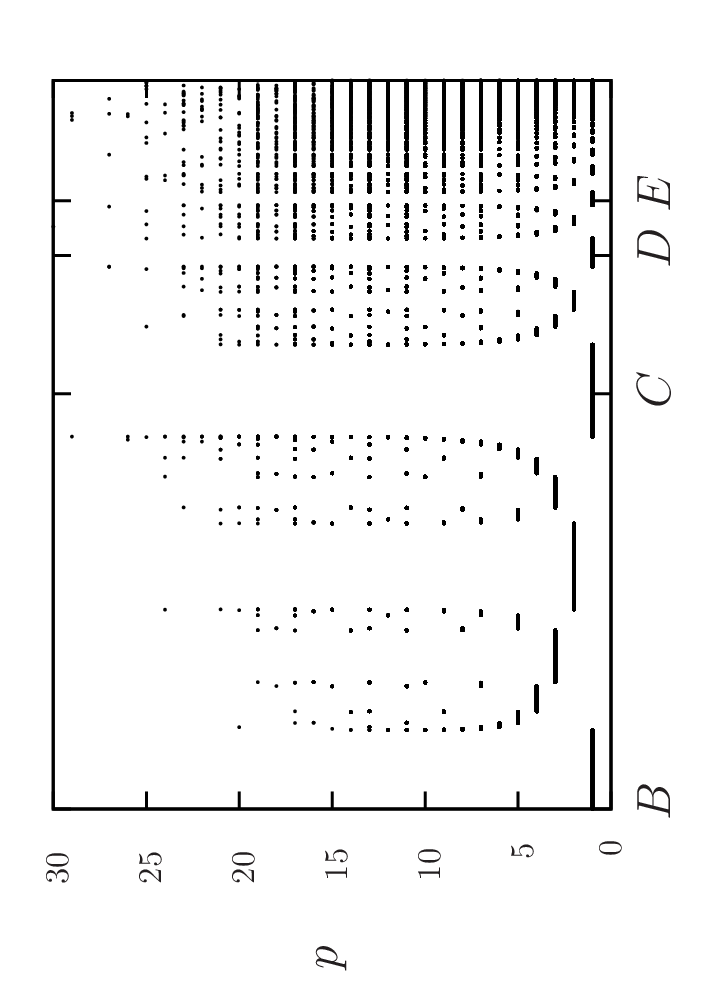}}
}
\end{picture}
\end{center}
\caption{(a) Bifurcation scenario for system~\SYSTEMWR{.} In gray regions there exist
$T$-periodic orbits, in white ones higher periodic orbits following and adding
structure (see text). In $B$, $C$, $D$ and $E$ one finds $T$-periodic orbits
spiking $0$, $1$, $2$ and $3$ times per period, respectively. (b) periods of the
periodic orbits found along the line shown in (a).}
\label{fig:d-invA_generic}
\end{figure}
The reset condition~\eqref{eq:reset_condition}, in combination with
conditions~\conds{}, implies that the dynamics we are interested in is located in
the interval $[0,\theta)$. In fact, by identifying $0\sim \theta$, this interval
can be seen as a circle. However, for our convenience, we will consider from now
on that the state space will be given by the interval $[0,\theta)$.\\
Note that the fact that the equilibrium point $\bx$ is globally attracting in
$[0,\theta)$ prevents the system to exhibit spikes without the input. It can
hence be considered as the limit of a slow/fast system.\\

One of the most relevant features of the asymptotic dynamics of
system~\SYSTEMWR{} is the \emph{firing rate}, i.e. the average number of spikes exhibited
by the system per unit time:
\begin{equation}\label{eq:defirate}
r(x_0)=\lim_{\tau\to\infty}\frac{\#(\mbox{spikes performed by } \phi(t;x_0)
\mbox{ for } t\in[0,\tau])}{\tau},
\end{equation}
where $\phi(t;x_0)$ is a trajectory of the system with $\phi(0;x_0)$.  When the
system possesses an attracting periodic orbit, this quantity does not depend on
the initial condition $x_0$ and becomes the number of spikes performed by this
periodic orbit along one period divided by the length of the period.\\
Hence, to obtain qualitative and quantitative description of the firing rate we
study the existence of periodic orbits for system~\SYSTEMWR{} and their
bifurcations under conditions~\conds{}. In particular, we provide a full
description of the bifurcation scenario for the two-dimensional parameter space
$d\times1/A$, which is shown in fig.~\ref{fig:d-invA_generic}. More precisely, under
conditions~\conds{} and for any $T>0$, we prove the following.
\begin{itemize}
\item The parameter space $d\times 1/A$, $d\in(0,1)$, contains an infinite
number of disjoint bounded regions (gray regions in
fig.~\ref{fig:d-invA_generic_regions}) for which the system~\SYSTEMWR{}
possesses a unique globally attracting $T$-periodic orbit (fixed point of the
stroboscopic map). These regions are clockwise ordered in such a way that the
$T$-periodic orbit corresponding to the next region exhibits one spike more per
period than the previous one (propositions~\ref{pro:existence_of_fp}
and~\ref{pro:1T_po}).
\item In between the regions corresponding to the existence of fixed points of
the stroboscopic map (grey regions) are the regions defined by the existence of
periodic orbits of arbitrarily large periods (white regions). These periodic
orbits are organized by the period adding structure
(proposition~\ref{pro:adding}).
\item Along any invertible and continuous curve
\begin{equation*}
\lambda\in\RR\longmapsto(d(\lambda),1/A(\lambda))
\end{equation*}
satisfying $d(\lambda)\in(0,1)$, $(1/A(\lambda))'<0$, $1/A(\lambda)\to0$ when
$\lambda\to\infty$ and crossing all the regions mentioned above (as the straight
line labeled in fig.~\ref{fig:d-invA_generic_regions}), the firing rate evolves
following a devil's staircase from $0$ to $\infty$. This is explained in
section~\ref{sec:symbolic}, and is a consequence of proposition~\ref{pro:adding}.
\end{itemize}

The main tool that we will use to prove the previous rsults is  the stroboscopic
map at return time $T$. As it will be shown, this is a piecewise-smooth map
defined on an infinite number of domains containing initial conditions for which
the system~\SYSTEMWR{} exhibits different number of spikes when flowed for a
time $T$. For each domain, there exist parameter values $d$ and $A$ for which
this map has a fixed point located in this domain; these parameter values
correspond to the gray regions in fig.~\ref{fig:d-invA_generic}.

Between two consecutive regions where fixed points exist (white regions in
fig.~\ref{fig:d-invA_generic_regions}), the stroboscopic map is discontinuous at
the boundary between two consecutive domains, and can be written as the normal
form map
\begin{equation}
g(x)=\left\{
\begin{aligned}
&\mu_\LL+g_\LL(x)&&\text{if }x<0\\
&-\mu_\R+g_\R(x)&&\text{if }x>0,
\end{aligned}\right.
\label{eq:normal_form}
\end{equation}
with $g_\LL,g_\R$ smooth functions satisfying $g_\LL(0)=g_\R(0)=0$. As the
stroboscopic map is given by the integration of a flow, its fixed points have
positive associated eigenvalues and hence $g_\LL(x)$ and $g_\R(x)$ are
increasing functions near the origin.  Moreover, due to the assumption $f'<0$,
the flow is everywhere contracting, so that $|g_\LL'|<1$ and $|g_\R'|<1$.

\unitlength=\textwidth
\begin{figure}
\begin{center}
\begin{picture}(1,0.5)
\put(0,0.35){
\subfigure[\label{fig:adding_regions}]{
\includegraphics[angle=-90,width=0.5\textwidth]{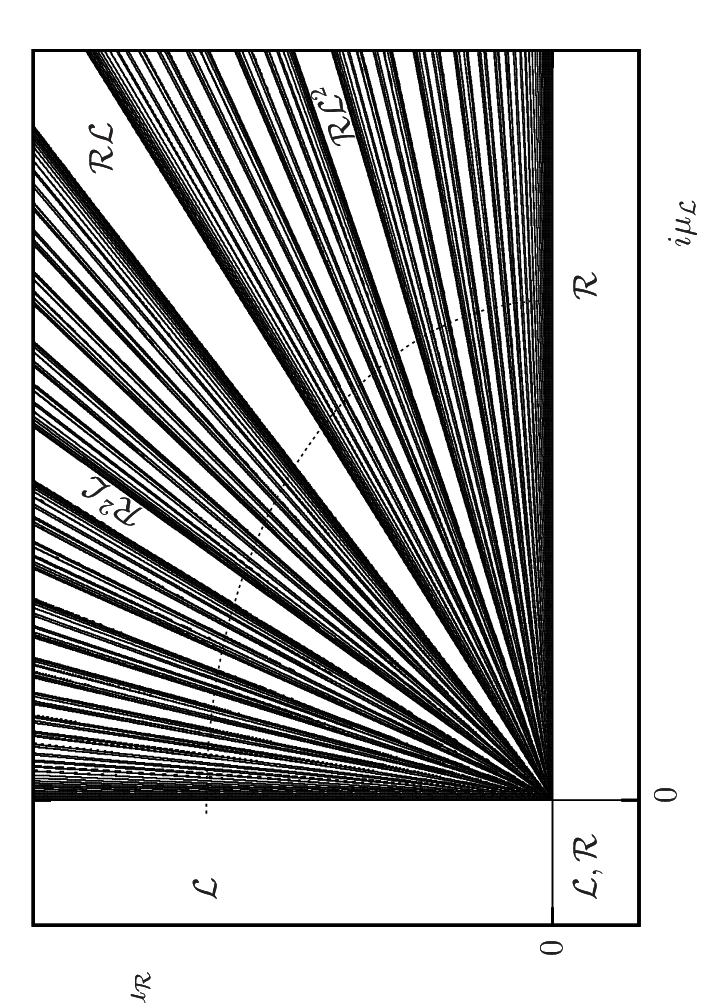}}
}
\put(0.5,0.35){
\subfigure[\label{fig:adding_periods}]{
\includegraphics[angle=-90,width=0.5\textwidth]{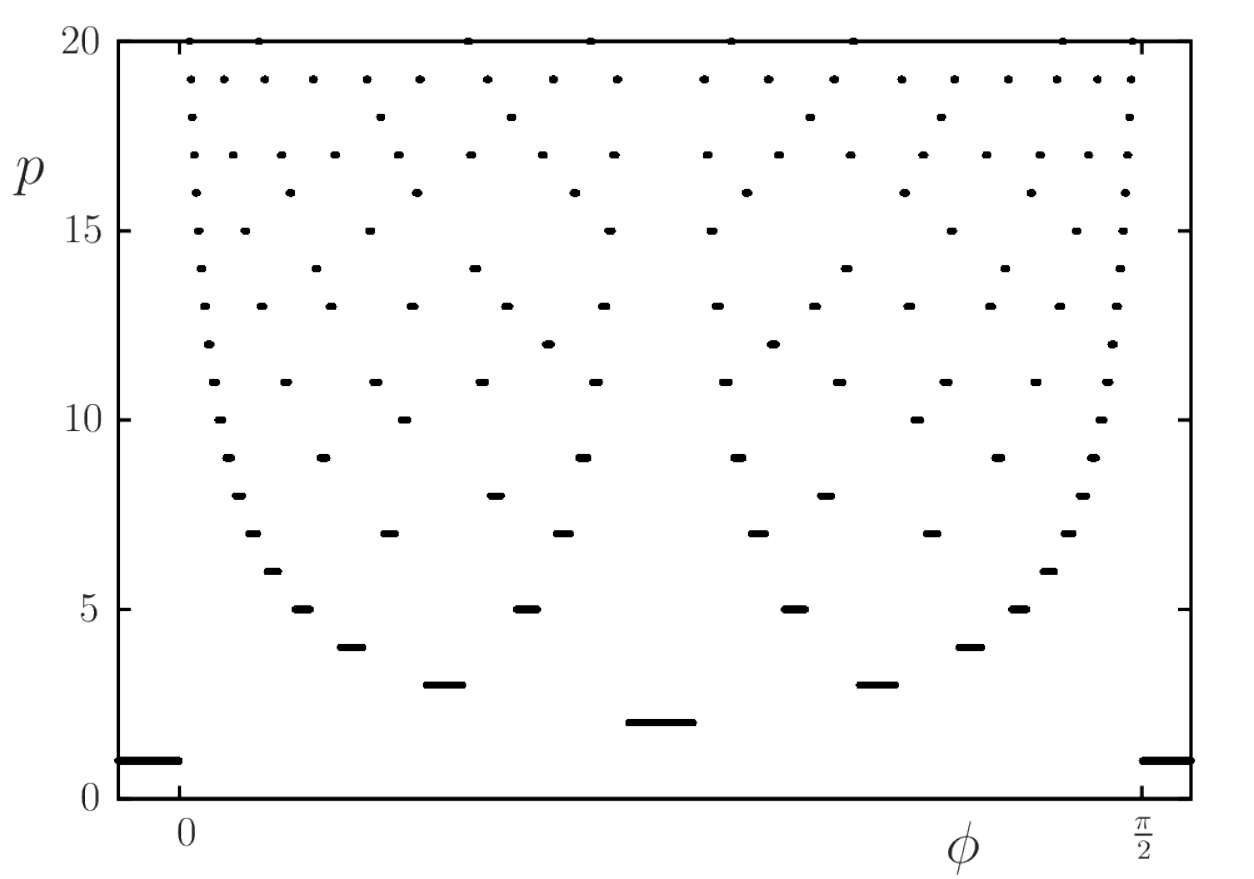}}
}
\end{picture}
\end{center}
\caption{The period adding big bang or gluying bifurcation. (a) $\mu_\LL\times\mu_\R$
parameter space, (b) periods of the periodic orbits along the curve shown in
(a).}
\label{fig:adding_bb}
\end{figure}

We now describe the bifurcation scenario for the map~\eqref{eq:normal_form},
which is shown in fig.~\ref{fig:adding_bb}.\\
Note that, for $\left| \mu_\LL \right|,\left| \mu_\R \right|$ small enough, if
$\mu_\LL<0$ and $\mu_\R<0$ the map~\eqref{eq:normal_form} possesses two fixed
points, each at each side of the boundary $x=0$, which undergo border
collision bifurcations when $\mu_\LL=0$ and $\mu_\R=0$, respectively. Hence, for
$\mu_\LL=\mu_\R=0$ both fixed points simultaneously collide with the boundary,
undergo border collision bifurcations and no longer exist. For
$\mu_\LL,\mu_\R>0$ there exist an infinite number of border collision
bifurcation curves separating regions of existence of periodic orbits with
arbitrarily high periods. These are organized by the so-called \emph{period
adding} phenomenon. To explain this, we first introduce the usual symbolic
dynamics.\\
Let us assume that $(x_1,..,x_n)$ is a $n$-periodic orbit of the
map~\eqref{eq:normal_form}. To this periodic orbit we assign a symbolic encoding
given by
\begin{equation}
\begin{aligned}
&x_i\to \LL&&\text{if }x_i<0\\
&x_i\to\R&&\text{if }x_i>0.
\end{aligned}
\label{eq:encoding}
\end{equation}
For each periodic orbit we obtain a symbolic sequence of $\LL$'s and
$\R$'s depending on whether the orbit steps on the left or on the right of the
boundary. Obviously, the length of the symbolic sequence is the period of the
periodic orbit.\\
Then, the period adding structure is described as follows. Between the regions
of existence of periodic orbits with symbolic sequences $\sigma_1$ and
$\sigma_2$ and periods $n_{\sigma_1}$ and $n_{\sigma_2}$ there exists a region
corresponding to the existence of a periodic orbit with symbolic sequence
$\sigma_1\sigma_2$. Clearly, the period of this periodic orbit is
$n_{\sigma_1}+n_{\sigma_1}$. This is occurs ad infinitum as illustrated in
fig.~\ref{fig:adding_periods}, where one can see the periods of the periodic
orbits obtained when crossing all these regions along a curve as the one shown
in fig.~\ref{fig:adding_regions}.\\
Beyond the evolution of the periods of the periodic orbits, one of the most
relevant quantities associated with the adding scheme is the \emph{rotation
number} corresponding to the periodic orbits obtained when crossing all these
regions in the parameter space. As shown in fig.~\ref{fig:farey_tree}, the
rotation number associated to each periodic orbit is obtained by the dividing
the number of $\R$ symbols contained in its symbolic sequence by the period of
the periodic orbit (see~\cite{GamGleTre84,GamLanTre84}). Hence, the rotation
number follows a \emph{devil's staircase} from $0$ to $1$ when para\-meters are
varied along curves as the one shown in fig.~\ref{fig:adding_regions}. A devil's
staircase is a monotonically increasing continuous function which is constant
almost every\-where, except for a Cantor set of zero measure. In this Cantor set,
which consists of the border collision bifurcation curves shown in
fig.~\ref{fig:adding_regions}, the rotation number becomes irrational and no
periodic orbit exists but a quasi-periodic one.

\begin{figure}
\begin{center}
\begin{picture}(1,0.5)
\put(0,0.015){
\subfigure[\label{fig:farey_tree}]{
\begin{picture}(0.5,0.5)
\put(-0.01,0.06){
\scalebox{0.65}{
\begin{tikzpicture}[->,>=stealth',shorten >=1pt,auto,node distance=1cm]
\node[] (a1){};
\node[] (a2) [right of=a1] {};
\node[] (a3) [right of=a2] {};
\node[] (a4) [right of=a3] {$\LL/0$};
\node[] (a5) [right of=a4] {};
\node[] (a6) [right of=a5] {$\R/1$};
\node[] (a7) [right of=a6] {};
\node[] (a8) [right of=a7] {};
\node[] (a9) [right of=a8] {};
\node[] (b1) [above of=a1] {};
\node[] (b2) [right of=b1] {};
\node[] (b3) [right of=b2] {};
\node[] (b4) [right of=b3] {};
\node[] (b5) [right of=b4] {$\LL\R/\frac{1}{2}$};
\node[] (b6) [right of=b5] {};
\node[] (b7) [right of=b6] {};
\node[] (b8) [right of=b7] {};
\node[] (b9) [right of=b8] {};
\node[] (c1) [above of=b1] {};
\node[] (c2) [right of=c1] {};
\node[] (c3) [right of=c2] {$\LL^2\R/\frac{1}{3}$};
\node[] (c4) [right of=c3] {};
\node[] (c5) [right of=c4] {};
\node[] (c6) [right of=c5] {};
\node[] (c7) [right of=c6] {$\LL\R^2/\frac{2}{3}$};
\node[] (c8) [right of=c7] {};
\node[] (c9) [right of=c8] {};
\node[] (d1) [above of=c1] {};
\node[] (d2) [right of=d1] {$\LL^3\R/\frac{1}{4}$};
\node[] (d3) [right of=d2] {};
\node[] (d4) [right of=d3] {};
\node[] (d5) [right of=d4] {};
\node[] (d6) [right of=d5] {};
\node[] (d7) [right of=d6] {};
\node[] (d8) [right of=d7] {$\LL\R^3/\frac{3}{4}$};
\node[] (d9) [right of=d8] {};
\node[] (e1) [above of=d1] {$\LL^4\R/\frac{1}{5}$};
\node[] (e2) [right of=e1] {};
\node[] (e3) [right of=e2] {};
\node[] (e4) [right of=e3] {$\LL^2\R\LL\R/\frac{2}{5}$};
\node[] (e5) [right of=e4] {};
\node[] (e6) [right of=e5] {$\LL\R\LL\R^2/\frac{3}{5}$};
\node[] (e7) [right of=e6] {};
\node[] (e8) [right of=e7] {};
\node[] (e9) [right of=e8] {$\LL\R^4/\frac{4}{5}$};
\path[]
(a4) edge node [left] {} (b5)
(a6) edge node [left] {} (b5)
(b5) edge node [left] {} (c3)
(a4) edge node [left] {} (c3)
(b5) edge node [left] {} (c7)
(a6) edge node [] {} (c7)
(c7) edge node [] {} (d8)
(a6) edge node [] {} (d8)
(a4) edge node [left] {} (d2)
(c3) edge node [left] {} (d2)
(a4) edge node [left] {} (e1)
(d2) edge node [left] {} (e1)
(a6) edge node [] {} (e9)
(d8) edge node [] {} (e9)
(b5) edge node [] {} (e4)
(c3) edge node [] {} (e4)
(b5) edge node [] {} (e6)
(c7) edge node [] {} (e6)
;
\end{tikzpicture}
}
}
\end{picture}
}
}
\put(0.5,0.35){
\subfigure[\label{fig:devils_staircase}]{
\includegraphics[angle=-90,width=0.5\textwidth]{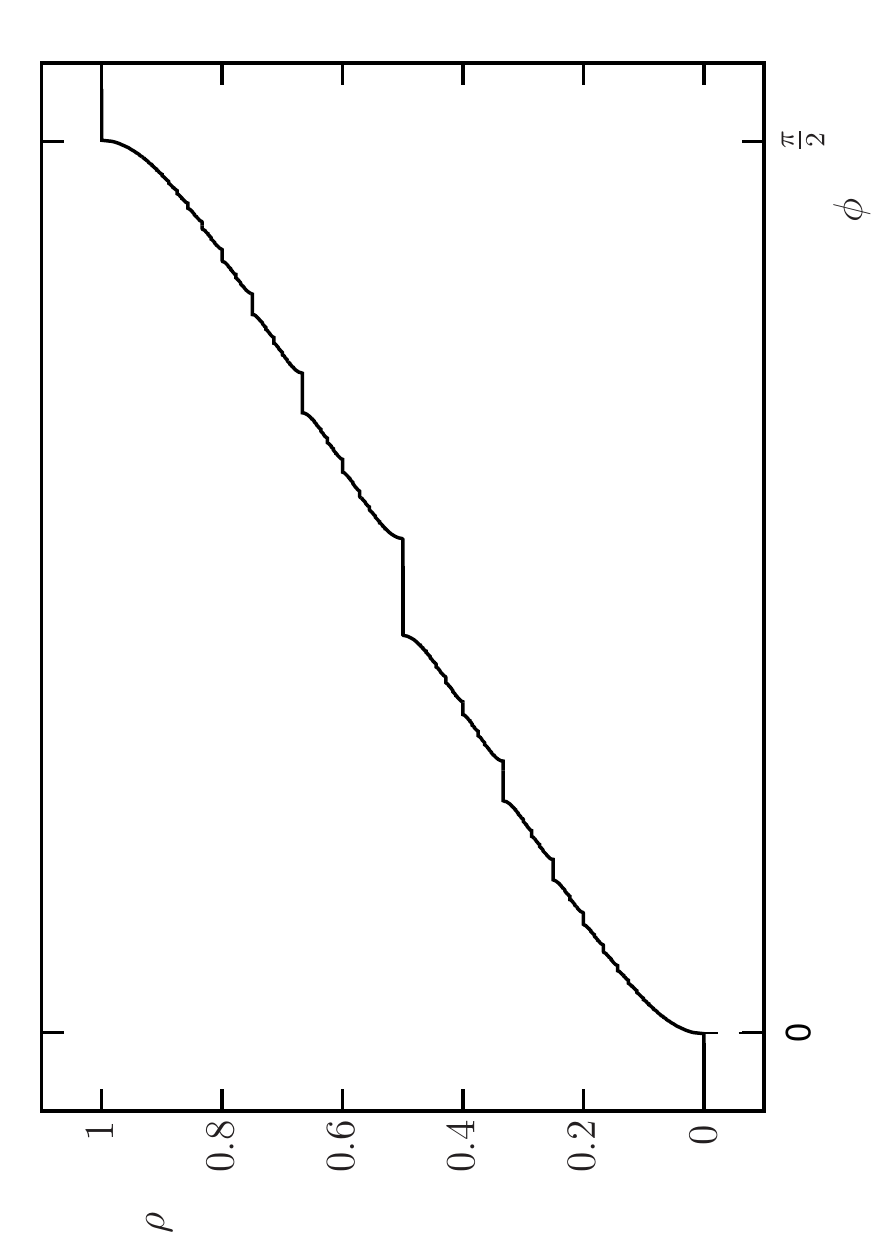}
}
}
\end{picture}
\caption{(a) Symbolic sequences and rotation numbers forming a Farey tree
structure for the period adding. (b) Devil's staircase associated with  the
rotation numbers along the Farey tree, obtained obtained along a curve as the
one labeled in fig.~\ref{fig:adding_regions}.}
\end{center}
\end{figure}
The period adding phenomenon described above was first described\linebreak
in~\cite{GamGleTre84}, and later studied in more detail
in~\cite{GamProThoTre86,LyuPikZak89,TurShi87,ProThoTre87,Gam87,GamTre88}.
Subsequently, this was rediscovered in~\cite{AvrSch06} when numerically
investigating the map~\eqref{eq:normal_form} with $g_\LL$ and $g_\R$ linear
functions. The authors of~\cite{AvrSch06} provided an accurate description of
the bifurcation scenario in terms of non-smooth language (border collision
bifurcations). A special interest was given to the codimension-two bifurcation
at the origin of the $\mu_\LL\times\mu_\R$ parameter space, which was called
\emph{period-adding big bang} bifurcation, to distinguish it from the so-called
\emph{period-incrementing big bang} bifurcation, being both a special type of
the \emph{gluing} bifurcation reported
in~\cite{CouGamTre84,GamGleTre88}. The period incrementing case, which exhibits
a completely different bifurcation scenario, occurs when one of the functions
$g_\LL$ or $g_\R$ is decreasing and the other one increasing near the origin
(being both contracting), as was rigorously proven in~\cite{AvrGraSch11}.
The result in the adding case (relevant to this paper) is stated below. A
complete proof of this result has not been published, but it can be pieced
together from partial results presented in \cite{Gam87,GamTre88,GamGleTre88}.

\begin{theorem}\label{theo:adding}
Let $g$ be a map as in eq.~\eqref{eq:normal_form} such that $g_\LL$ and
$g_\R$ are increasing contracting functions,
\begin{equation}
\begin{aligned}
&0<(g_\LL(x))'<1&& x\in(-\infty,0)\\
&0<(g_\R(x))'<1&&x\in(0,\infty)
\end{aligned}
\label{eq:increasing_contracting_condition}
\end{equation}
Consider a $C^1$ curve in the parameter space $\mu_\LL\times\mu_\R$
\begin{equation*}
\begin{array}{cccc}
\gamma:&[0,1]&\longrightarrow&\RR^2\\
&\lambda&\longmapsto &(\mu_\LL(\lambda),\mu_\R(\lambda))
\end{array}
\end{equation*}
satisfying
\begin{enumerate}[C.1]
\item $\mu_\LL(\lambda)> 0$ and $\mu_\R(\lambda)>0$ for $\lambda\in(0,1)$
\item $\left( \mu_\LL(\lambda) \right)'>0$ and $\left( \mu_\R(\lambda) \right)'<0$
for $\lambda\in[0,1]$
\item $\mu_\LL(0)=0$, $\mu_\R(1)=0$
\end{enumerate}
Then,
%
the bifurcation diagram exhibited by the map $g_\lambda$ obtained from
eq.~\eqref{eq:normal_form} after performing the reparametrization given by
$\gamma$, follows a period adding strucutre for $\lambda\in[0,1]$. Moreover, the
rotation number $\rho(g_\lambda)$ of the map $g_\lambda$ is a devil's staircase
as a function of $\lambda$ such that $\rho(g_0)=0$ and $\rho(g_1)=1$.
\end{theorem}

\begin{remark}\label{rem:local_contractiveness}
The increasing and contractive
condition~\eqref{eq:increasing_contracting_condition} can be relaxed to be
satisfied locally near $x=0$ if $\mu_\LL(\lambda)$ and $\mu_\R(\lambda)$ are
small enough for $\lambda\in(0,1)$. This is because the periodic orbits are
located in the aboserving interval $[-\mu_\R(\lambda),\mu_\LL(\lambda)]$.
\end{remark}

By means of theorem~\ref{theo:adding} we will show that the stroboscopic map
follows period adding bifurcation structures when parameters are varied along
curves as the straight line shown in fig.~\ref{fig:d-invA_generic_regions}: the
white regions in fig.~\ref{fig:d-invA_generic} surrounded by gray ones are
filled by the bifurcation structres shown in fig.~\ref{fig:adding_bb}.
We will use the properties of the devil's staircase of the rotation number to
deduce properties of the firing rate defined in
eq.~\eqref{eq:defirate}.

\section{Dynamics of the stroboscopic map}\label{sec:proof_of_results}

\subsection{Notation, properties and subthreshold dynamics}\label{sec:system_description}
Let
\begin{equation}
\varphi(t;x_0;A)
\label{eq:flow_A}
\end{equation}
be the flow associated with system
\begin{equation}
\dot{x}=f(x)+A
\label{eq:autonomous_system_d1}
\end{equation}
such that $\varphi(0;x_0;A)=x_0$. As usual in piecewise-defined systems, by
properly concatenating the flows $\varphi(t;x_0;A)$, $\varphi(t;x_0;0)$ and the
reset condition~\eqref{eq:reset_condition} one obtains the flow
\begin{equation}
\phi(t;x_0)
\label{eq:flow_non-auto}
\end{equation}
associated with system~\SYSTEMWR{} such that $\phi(0;x_0)=x_0$. Due to the reset
conditions this flow is discontinuous at the times for which spikes occur. In
addition, due to the discontinuities of the forcing $I(t)$ defined
in~\eqref{eq:forcing}, the flow $\phi$ is non differentiable at
$t=dT\mod{T}$.\\
Note that for the definition of the flow $\phi$ associated with the
non-autonomous system~\eqref{eq:general_system} we consider that the initial
condition $x_0$ occurs at $t_0=0$. This will be assumed in the rest of this work
and, as we are interested on the existence of periodic orbits, does not present
any loss generality.\\

Under assumptions~\conds{} we first study the subthreshold dynamics; that is,
invariant objects of system~\SYSTEMWR{} that do not interact with the threshold
(do not exhibit spikes).\\

On one hand, if $d=0$, then $I(t)=0$ and the system~\SYSTEMWR{} has
the same dynamics as the autonomous system~(\ref{eq:autonomous_system}). That
is, it has an attracting equilibrium point at $x=\bx$.

On the other hand, if $d=1$, then $I(t)=A$. Hence system~\SYSTEMWR{} becomes the
autonomous system~\eqref{eq:autonomous_system_d1} subject to the reset
condition~\eqref{eq:reset_condition} and, as for $d=0$, the $T$-periodic forcing
does not play any role. By the implicit function theorem, if $A>0$ is small
enough the system possesses an attracting equilibrium point at
\begin{equation*}
x^*=\bx-\frac{A}{f'(\bx)}+O(A^2)<\theta.
\end{equation*}
Note that, as $f(x)$ is a monotonically decreasing function (condition H.2,) and
hence this equilibrium point increases with $A$ towards the boundary $x=\theta$.
Thus, for $d=1$, when $A$ is large enough this equilibrium point collides with
the boundary $x=\theta$, undergoes a border collision bifurcation. After that,
the system has a periodic orbit that spikes once per period. The period of this
periodic orbit, $\delta>0$, becomes the time needed by system $\dot{x}=f(x)+A$
with initial condition $x_0=0$ to reach the threshold $x=\theta$. This is the
smallest $\delta>0$ such that
\begin{equation}
\varphi(\delta;0;A)=\theta,
\label{eq:delta}
\end{equation}
if it exists ($A$ is large enough).\\
Let us now study the subthreshold dynamics of the system when $0<d<1$. We have
the next
%
\begin{prop}\label{prop:periodic_orbit}
Let
\begin{equation*}
Q:=\frac{1}{T}\int_0^T I(t)dt=Ad
\end{equation*}
be small enough. Then, if $T>0$ is small enough,
system~(\ref{eq:general_system})-(\ref{eq:reset_condition}) has a $T$-periodic
orbit that does not hit the boundary $x=\theta$.
\end{prop}
\begin{proof}
By averaging the system. After the time rescale $\tau=\frac{t}{T}$
system~(\ref{eq:general_system}) becomes
\begin{equation}
\dot{x}=T\left(f(x)+I(\tau T)\right),
\label{eq:general_system_rescaled}
\end{equation}
where $\dot{}$ means now derivative with respect to $\tau$ and $I(\tau T)$ is
$1$-periodic. We now consider the averaged version of
system~(\ref{eq:general_system_rescaled}),
\begin{equation}
\dot{y}=T\left( f(x)+Q \right).
\label{eq:averaged_system}
\end{equation}
By the averaging theorem, as system~\eqref{eq:general_system_rescaled} is
Lipschitz in $x$, if system~(\ref{eq:averaged_system}) possesses a hyperbolic
equilibrium point then, for $T>0$ small enough,
system~(\ref{eq:general_system_rescaled}) possesses a  hyperbolic periodic orbit
(see~\cite{BogMit61,GucHol83}). This will occur if $Q>0$ is small enough, and
the equilibrium point of~(\ref{eq:averaged_system}) will be of the form
\begin{equation*}
\bar{y}=\bx-\frac{Q}{f'(\bx)}+O\left(  Q^2 \right)<\theta,
\end{equation*}
where $\bx$ is the hyperbolic equilibrium point of
system~(\ref{eq:autonomous_system}).\\
Then, the periodic orbit of system~(\ref{eq:general_system_rescaled}) will be
$T$-close to $\bar{y}$.
\end{proof}

When $A$ is large, the periodic orbit given by the previous lemma may undergo a
border collision bifurcation and lead to spiking dynamics. This new dynamics
will be studied in sections~\ref{sec:bif_struct} and~\ref{sec:period_adding} by
means of the stroboscopic map, i.e. the time-$T$ return map
\begin{equation}
\begin{array}{cccc}
\s:&[0,\theta)&\longrightarrow &[0,\theta)\\
&x_0&\longmapsto&\phi(T;x_0).
\end{array}
\label{eq:strobo_map}
\end{equation}
The main results in this work rely on the fact that $\s$ is a piece-smooth map.
To see this, we define the sets
\begin{equation}
\begin{aligned}
S_n=\Big\{ x_0\in [0,\theta)\;\text{s.t. }&\phi(t;x_0)\text{ reaches
 the threshold }\left\{ x=\theta\right\}\\
&n\text{ times for }0\le t\le T \Big\},\, n\ge0.
\end{aligned}
\label{eq:sets}
\end{equation}
When restricted to some set $S_n$, the flow $\phi(t;x_0)$ becomes a certain
combination of the smooth flows $\varphi(t;x_0;A)$, $\varphi(t;x_0;0)$ and the
smooth mapping $\theta\mapsto0$ (given by the reset condition) which does not
depend on the initial condition $x_0\in S_n$. Hence, the stroboscopic map $\s$
becomes a smooth map in each of the sets $S_n$ and is as regular as the flow
$\varphi(t;A)$.  Given $m\neq n$, the flow, when restricted to $S_n$ and $S_m$,
performs a different number of spikes in the time window $[0, T]$. Hence $\s$ is
given by a different combination of the mentioned maps.  Thus, $\s$ is smooth on
the interiors of $S_n$ and $S_m$ but will typically be discontinuous at its boundaries.

Note that some of the sets $S_n$ may be empty. The next lemma tells us that at
most two of them are non empty and they must be consecutive (see
fig.~\ref{fig:boundary} for $S_2$ and $S_3$). Moreover, let $\Sigma_n\in S_n$ be
defined by the requirement
\begin{equation}
\phi(dT;\Sigma_n)=\theta.
\label{eq:boundaries}
\end{equation}
That is, the (unique) initial condition whose trajectory spikes $n$ times and
performs its last spike precisely at $t=dT$. The next lemma tells us also that
$\Sigma_n$ is the value at which the map $\s$ exhibits a discontinuity and
separates the sets $S_{n-1}$ and $S_n$.

\begin{lem}\label{lem:unique_Sigma_n}
Assume that there exists $\Sigma_n\in S_n$ as given in
eq.~\eqref{eq:boundaries} for some $n\ge1$.
Then the following statements hold:
\begin{enumerate}[i)]
\item If $\Sigma_i,\Sigma_j\in[0,\theta)$, then $i=j$.
\item If $\Sigma_n\in(0,\theta)$, then $[0,\Sigma_n)=S_{n-1}$ and
$[\Sigma_n,\theta)=S_n$.
\item If $\Sigma_n=0$, then $[0,\theta)=S_n$.
\end{enumerate}
\end{lem}

\begin{figure}
\begin{center}
\begin{picture}(1,0.5)
\put(0,0.35){
\subfigure[]
{\includegraphics[angle=-90,width=0.5\textwidth]{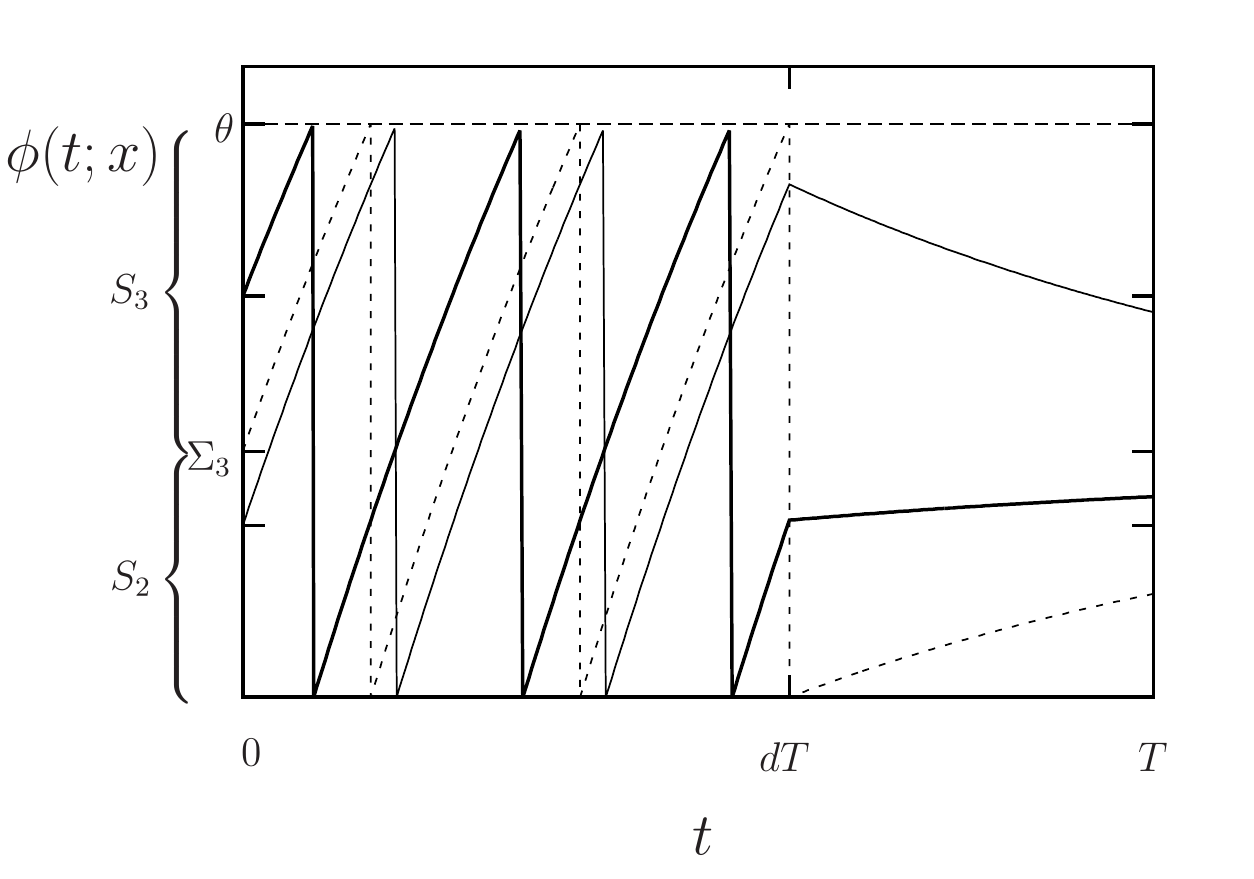}}
}
\put(0.5,0.35){
\subfigure[\label{fig:boundary_map}]
{\includegraphics[angle=-90,width=0.5\textwidth]{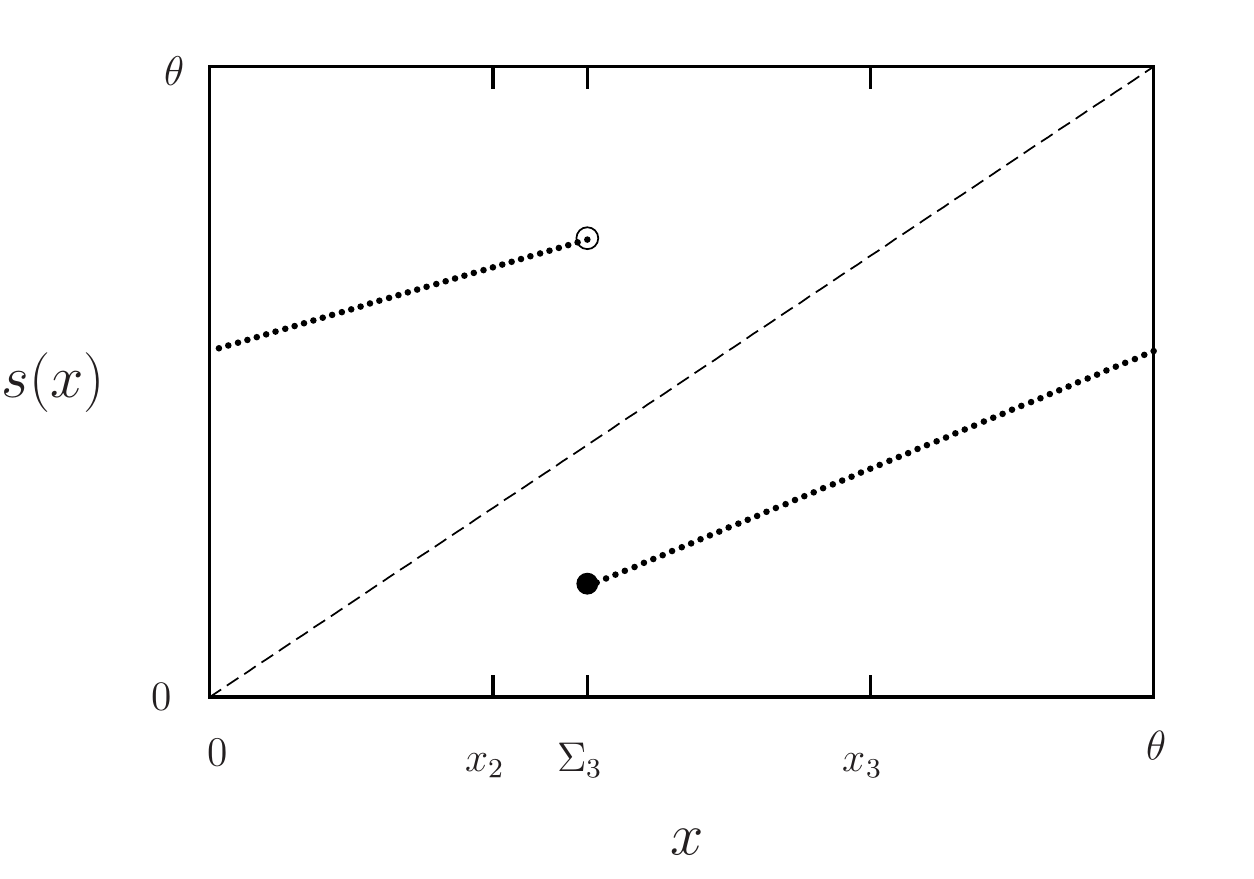}}
}
\end{picture}
\end{center}
\caption{In (a) the trajectories of systems~\SYSTEMWR{}. Dashed line: trajectory
with $\Sigma_3$ as initial condition. Thick line: trajectory with $x_3>\Sigma_3$
as initial condition, which spikes $3$ times. Normal line: trajectory with
$x_2<\Sigma_3$ as initial condition, which spikes $2$ times. In (b) the
stroboscopic map, with a discontinuity at $x=\Sigma_3$.}
\label{fig:boundary}
\end{figure}

\begin{proof}
The first part \emph{i)} comes from the fact that the flow $\phi$ is invertible.
Hence, if it exists, the initial condition that makes $\phi$ reach $\theta$ at
$t=dT$ is unique.\\
We now see \emph{ii)}. Let
\begin{equation}
t_n<t_n+\delta<\dots<t_n+(n-2)\delta<t_n+(n-1)\delta=dT
\label{eq:spiking_times}
\end{equation}
be the sequence of times for which the trajectory $\phi(t;\Sigma_n)$ exhibits
spikes.  Let us now take an initial condition $x_0<\Sigma_n$.  For $t<dT$ the
flow $\phi(t;x_0)$ becomes the flow of the autonomous system $\dot{x}=f(x)+A$
plus the reset condition~\eqref{eq:reset_condition}. Hence, the first spike of
the trajectory $\phi(t;x_0)$ occurs for $t>t_n$.  By induction, the $n$th spike
for such a trajectory occurs for $t\in(dT-\delta,dT)$, with $\delta$ as in
eq.~\eqref{eq:delta}. For $t>dT$ the trajectory $\phi(t;x_0)$ becomes the flow
of system $\dot{x}=f(x)$, which has an attractor $\bx\in(0,\theta)$; hence, no
other spikes can occur for $t>dT$ and the trajectory $\phi(t;x_0)$ exhibits
exactly $n-1$ spikes for $0\le t\le T$.  Using \emph{i)} we get that,
$[0,\Sigma_n)= S_{n-1}$. Arguing similarly, any trajectory $\phi(t;x_0)$ with
$x_0\in[\Sigma_n,\theta)$ exhibits exactly $n$ spikes for $0\le t\le T$ and thus
$[\Sigma_n,\theta)=S_n$.\\
Proceeding similarly, if $\Sigma_n=0$, then all initial conditions in
$[0,\theta)$ lead to trajectories exhibiting a $n$th spike for
$t\in(dT-\delta,dT)$.  Hence, $S_n=[0,\theta)$.
\end{proof}

\begin{remark}
The values $\Sigma_n$ become upper and lower boundaries of the sets
$S_{n-1}$ and $S_n$, respectively.
\end{remark}
The following results give us properties of the boundaries $\Sigma_n$ which we
will use in section~\ref{sec:bif_struct}.
\begin{lem}\label{lem:monotonicity}
Let $d\in(0,1)$, $A>0$ and assume $\Sigma_n\in(0,\theta)$. Then $\Sigma_n$ is a
monotonically decreasing function of $A$.
\end{lem}

\begin{proof}
To avoid confusions with the differential and derivative, in this proof we
rename the duty cycle $d$ by $a$.\\
The boundary $\Sigma_n$ is determined by the initial condition such that the
flow $\phi(t;x_0)$ collides with the boundary for the $n$th time at $t=aT$. Let
$\varphi(t;x_0;A)$ be the flow defined in eq.~\eqref{eq:flow_A}, and let
$\delta>0$ be the time needed by with initial condition $x_0=0$ to reach the
threshold $\theta$, as defined in eq.~\eqref{eq:delta}. Note that $\delta$ is a
decreasing function of $A$. In fact, rewriting eq.~\eqref{eq:delta} as
\begin{equation*}
\int_0^\theta\frac{dx}{f(x)+A}=\delta,
\end{equation*}
we get that
\begin{equation*}
\frac{d\delta}{dA}=-\int_0^\theta\frac{dx}{\left( f(x)+A \right)^2}<0.
\end{equation*}
The boundary $\Sigma_n$ is thus determined by the equation
\begin{equation*}
\varphi(aT-(n-1)\delta;\Sigma_n,A)=\theta,
\end{equation*}
or, equivalently, by the equation
\begin{equation*}
\int_{\Sigma_n}^\theta\frac{dx}{f(x)+A}=aT-(n-1)\delta.
\end{equation*}
Differentiating we get that
\begin{equation*}
\frac{d\Sigma_n}{dA}=-\left( f(\Sigma_n)+A \right)\left( \int_{\Sigma_n}^\theta
\frac{dx}{\left( f(x)+A \right)^2}-(n-1)\frac{d\delta}{dA} \right).
\end{equation*}
Noting that $f(\Sigma_n)+A>0$ because the system $\dot{x}=f(x)+A$ possesses an
attracting equilibrium point for $x>\theta$ (which permits to spike), we get
that
\begin{equation*}
\frac{d\Sigma_n}{dA}<0.
\end{equation*}
\end{proof}

The next lemma provides the lateral values of the stroboscopic map $\s$ at the
discontinuities.
\begin{lem}\label{lem:lateral_values}
Let
\begin{align}
\s_-&=\varphi(T(1-d);\theta;0)\label{eq-defsmin}\\
\s_+&=\varphi(T(1-d);0;0)\label{eq-defsplus},
\end{align}
where $\varphi(t;x_0;A)$ is the flow defined in~\eqref{eq:flow_A}.\\
If $\Sigma_n\in[0,\theta)$, then the lateral images of $\Sigma_n$ by $\s$ do not
depend on $A$ or $n$ and become
\begin{align}
\s(\Sigma_n^-)=\liminf_{x\to\Sigma_n}\s(x)&=\s_-\label{eq:leftimage}\\
\s(\Sigma_n^+)=\limsup_{x\to\Sigma_n}\s(x)&=\s_+\label{eq:rightimage},
\end{align}
\end{lem}
\begin{proof}
The trajectory of $\Sigma_n$ is such that the flow $\phi(t;\Sigma_n)$ reaches
the threshold for the $n$th time exactly at $t=dT$. The limiting value 
$\s(\Sigma_n^-)$ is obtained assuming that the flow does not spike at $t=dT$
(the threshold is almost reached). Hence, one just needs to integrate the flow
$\varphi(t;x_0;0)$ from $t=dT$ to $t=T$ with initial condition $x_0=\theta$.
This gives eq.~\eqref{eq:leftimage}.\\
On the other hand, the right image is computed adding a new spike at $t=dT$.
Hence, one has only to proceed identically but replacing the initial condition
by the one given by the reset, $x_0=0$.
\end{proof}

\subsection{Border collision bifurcations of spiking fixed points}\label{sec:bif_struct}
As noted in the previous section, non spiking periodic orbits (without
interaction with the threshold) exist for $A>0$ small enough and $d=0$ or $d=1$.
Hence, the natural parameter space where to study the possible bifurcations
leading to spiking dynamics is $d\times A$.  Equivalently, we will consider
$d\times 1/A$ instead. This is because, in this space, the bifurcation curves
will be bounded.

In order to show that the bifurcation scenario in the parameter space
$d\times1/A$ for system~\SYSTEMWR{} is equivalent to the one shown in
fig.~\ref{fig:d-invA_generic} and described in~\S\ref{sec:results}, we first
show the existence of an infinite number of regions for which $\s$ possesses
fixed points. These fixed points are located at different domains $S_n$ defined
in~\eqref{eq:sets}. These regions in the parameter space are the gray regions in
fig.~\ref{fig:d-invA_generic}. As shown in the next result, when they exist,
these fixed points are unique and are located at one side of some discontinuity
$\Sigma_n$. By varying parameter values, they bifurcate when entering the white
regions in figure~\ref{fig:d-invA_generic_regions}.\\
The next proposition provides the existence of unique fixed points, $\bx_n\in
S_n$ for any $n\ge0$.

\begin{prop}\label{pro:existence_of_fp}
Let $d\in(0,1)$ and $n\ge 0$. Then, there exists some value of $A$ for which the
stroboscopic map $\s$ defined in~\eqref{eq:strobo_map} possesses a fixed point
$\bx_n\in S_n$. When it exists, this fixed point becomes the only invariant
object of $\s$.
\end{prop}
\unitlength=\textwidth
\begin{figure}
\begin{center}
\begin{picture}(1,0.8)
\put(0,0.8){
\subfigure[\label{fig:T-per_R-bif}]
{\includegraphics[angle=-90,width=0.5\textwidth]{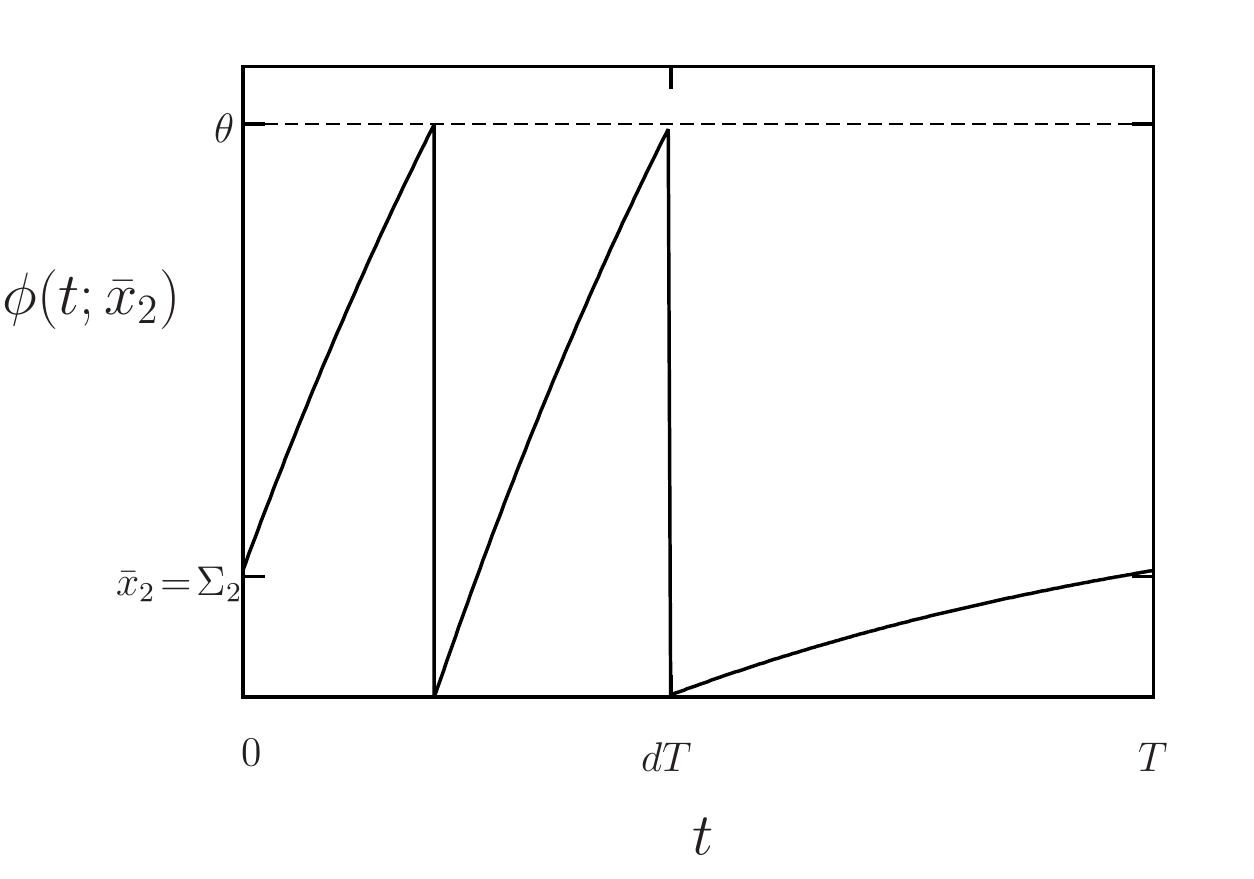}}
}
\put(0.5,0.8){
\subfigure[\label{fig:T-per_2}]
{\includegraphics[angle=-90,width=0.5\textwidth]{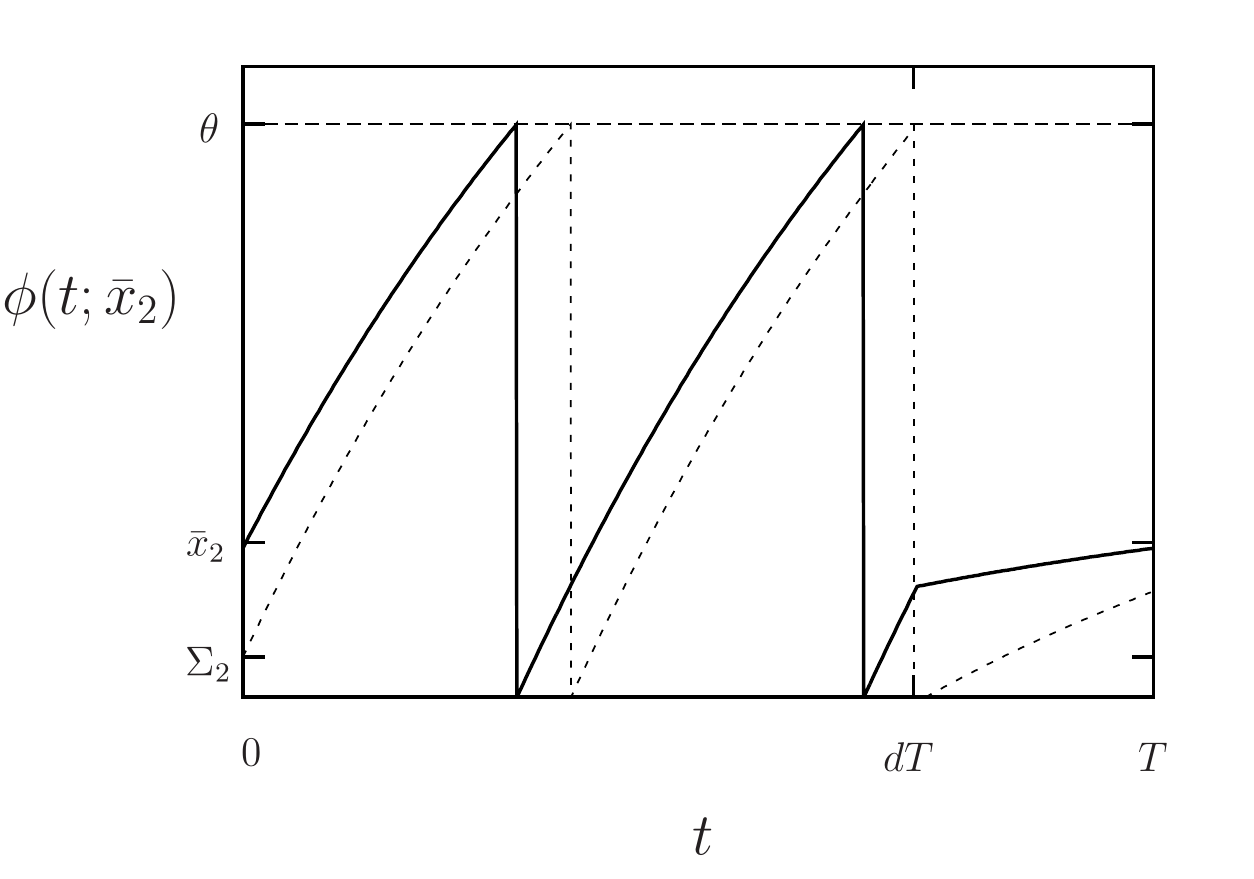}}
}
\put(0,0.4){
\subfigure[\label{fig:T-per_3}]
{\includegraphics[angle=-90,width=0.5\textwidth]{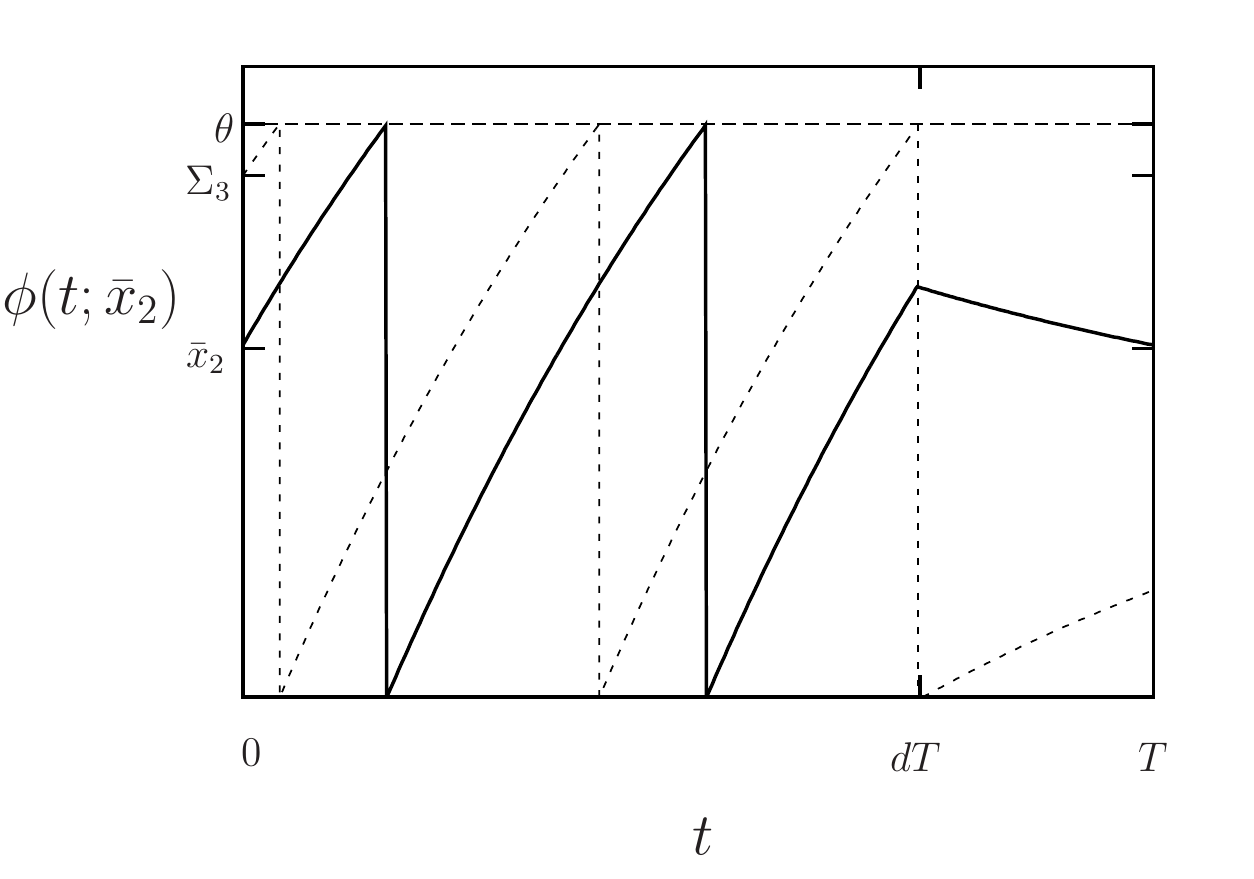}}
}
\put(0.5,0.4){
\subfigure[\label{fig:T-per_L-bif}]
{\includegraphics[angle=-90,width=0.5\textwidth]{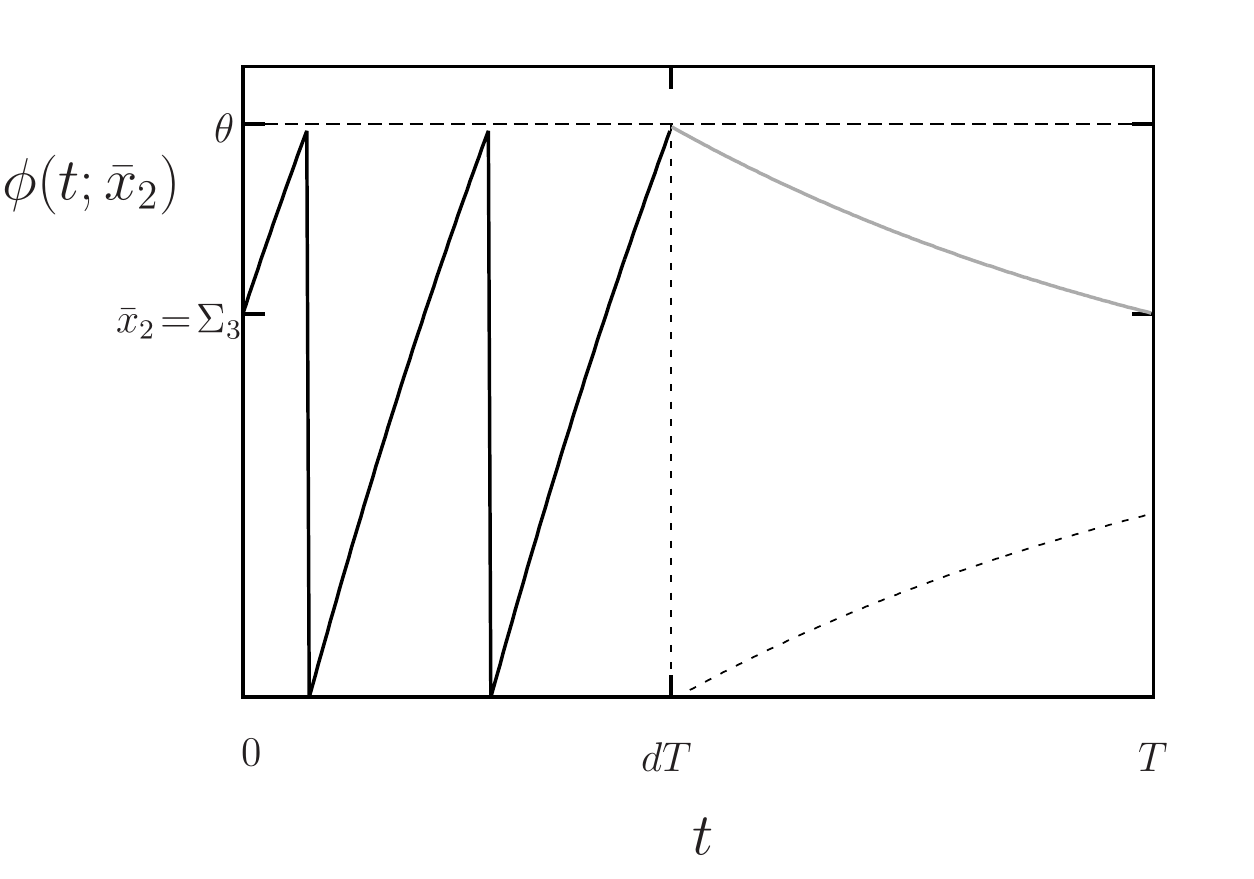}}
}
\end{picture}
\end{center}
\caption{$T$-periodic orbit spiking twice per period (fixed point $\bx_2$ of the
stroboscopic map) a its bifurcations when $A$ is varied along
$(A_2^\R(d),A_2^\LL(d))$. It undergoes border collision bifurcations when it
collides with the boundaries $\Sigma_2$ and $\Sigma_3$ (a) and (d),
respectively.  The periodic orbit shown in (d) is its limit when
$\bx_2\to\Sigma_3^-$; note that for $\bx_2=\Sigma_3$ it should be reset to $0$
at $t=dT$, this is why it is shown in gray. In (b) and (c), the trajectories of
these boundaries are shown in dashed lines; note that they collide with the
threshold at $t=dT$.  Parameter values for panel~(c) are the same as for point
$D$ of figure~\ref{fig:d-invA_generic_regions}.  The four figures are a in one
to one correspondence with the four figures of figure~\ref{fig:maps_bx2}, where
the stroboscopic map is shown.}
\label{fig:bif_T-per}
\end{figure}

\begin{figure}
\begin{picture}(1,0.8)
\put(0,0.8){
\subfigure[\label{fig:map_T_R-bif}]{\includegraphics[angle=-90,width=0.5\textwidth]
{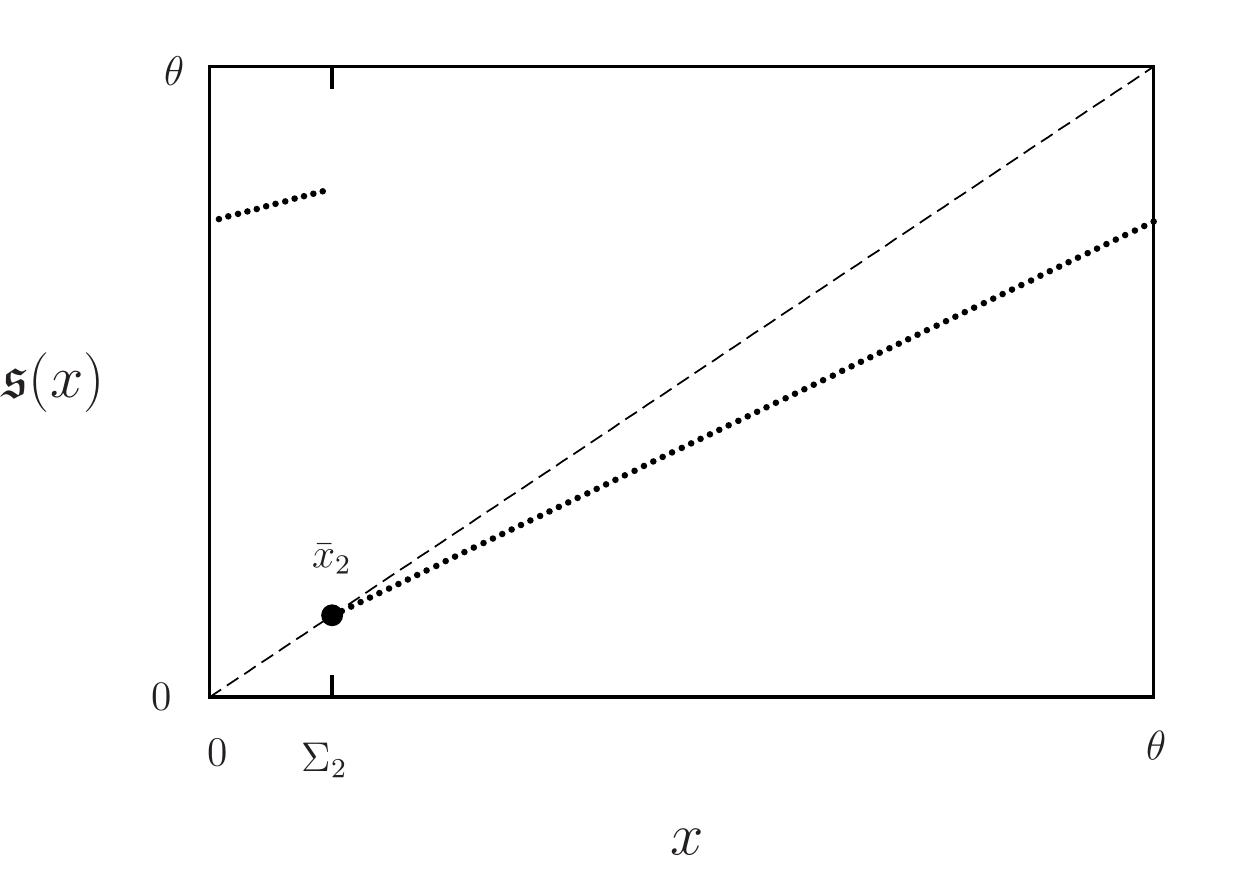}}
}
\put(0.5,0.8){
\subfigure[\label{fig:map2}]{\includegraphics[angle=-90,width=0.5\textwidth]
{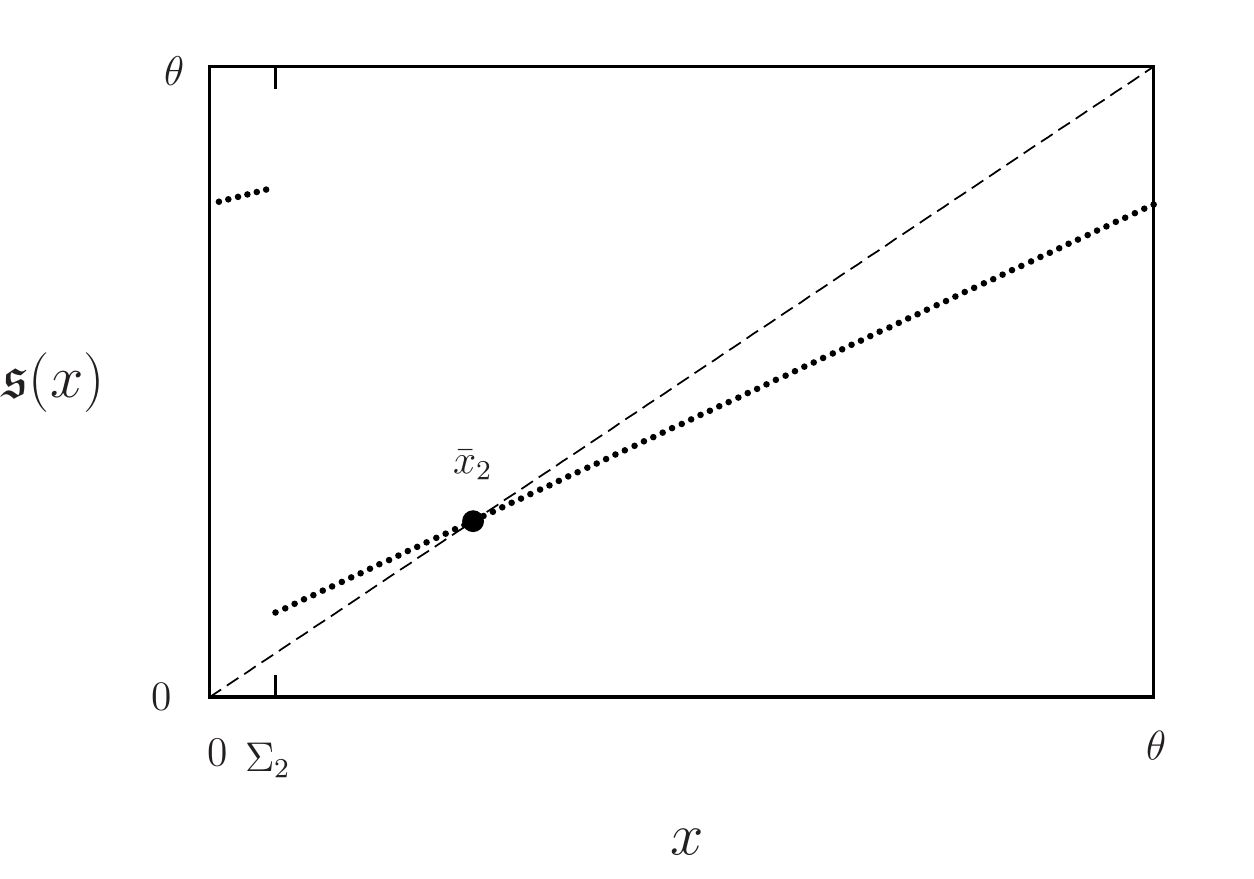}}
}
\put(0,0.4){
\subfigure[\label{fig:map3}]{\includegraphics[angle=-90,width=0.5\textwidth]
{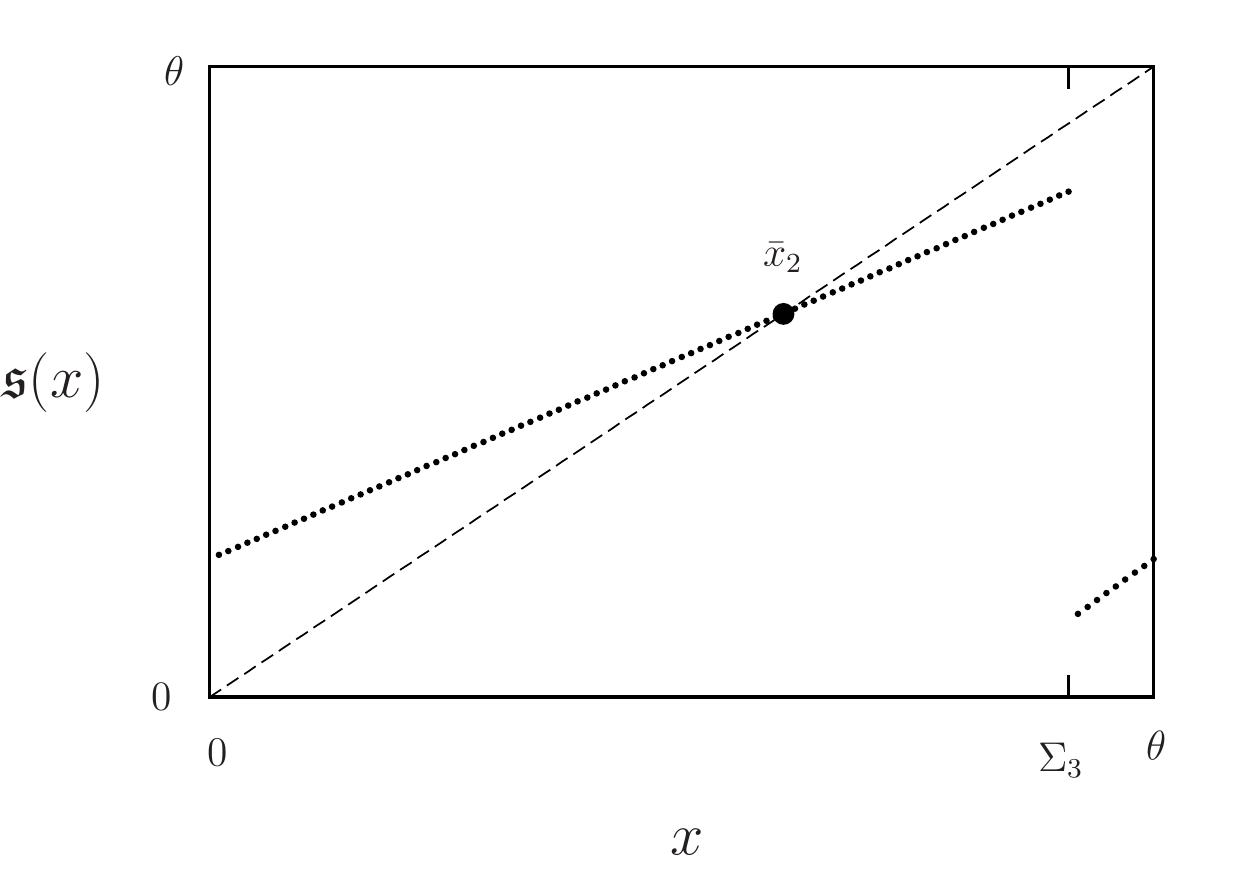}}
}
\put(0.5,0.4){
\subfigure[\label{fig:map_T_L-bif}]{\includegraphics[angle=-90,width=0.5\textwidth]
{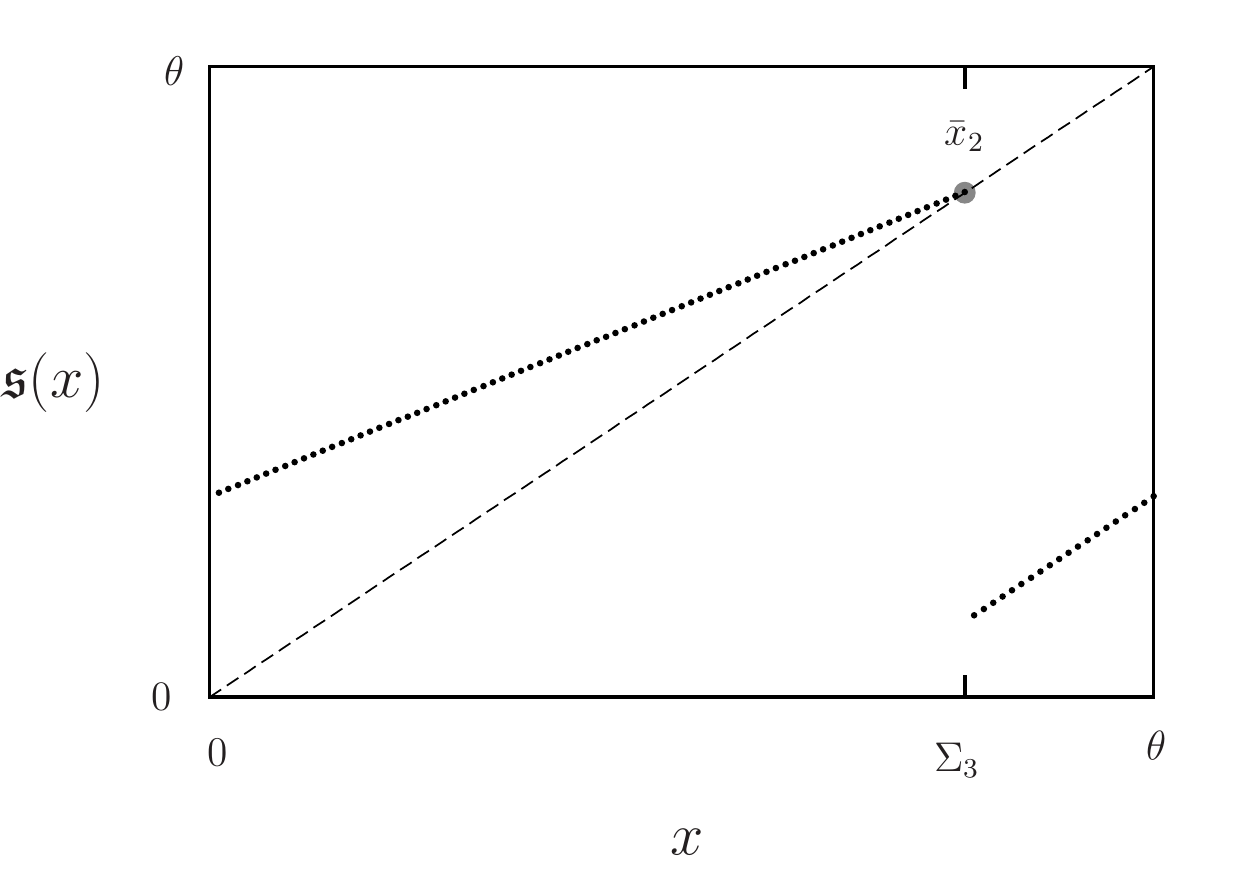}}
}
\end{picture}
\caption{Stroboscopic map for the $T$-periodic orbits shown in
figure~\ref{fig:bif_T-per}. In (a) and (d) the fixed point $\bx_2$ undergoes
border collision bifurcation when it collides with the boundaries $\Sigma_2$
from the right and $\Sigma_3$ from the left, respectively. Note that, in (d) the
fixed point is shown in gray to emphasize that the map takes indeed the value on
the right for $x=\Sigma_3$.  In (b)-(c) the boundary $\Sigma_2$ disappears and a
new boundary $\Sigma_3$ appears while the fixed point $\bx_2$ remains.}
\label{fig:maps_bx2}
\end{figure}
\begin{proof}
The existence of the fixed point $\bx_0\in S_0$ of the stroboscopic map
$\s$~\eqref{eq:strobo_map} for any $d\in(0,1)$ when $A>0$ is small enough comes
from proposition~\ref{prop:periodic_orbit}.
By induction we now show that, for any $n\ge0$,  there exist values of $A$ for
which one finds fixed points of $\s$, $\bx_n$, leading to periodic orbits
spiking $n$ times: $\bx_n\in S_n$. Assume that there exists a fixed point
$\bx_{n-1}\in S_{n-1}$. By if necessary increasing $A$, we can assume that
$S_{n-1}=[0,\Sigma_{n})$ ($\bx_{n-1}$ is located at the left of the
discontinuity). Then, by lemma~\ref{lem:monotonicity}, $\Sigma_{n}$ is a
monotonically decreasing function of $A$. Hence, by further increasing $A$, one
necessary finds some value, $A=A_{n}^{\cal C}(d)$ for which $\Sigma_{n}=0$. At
this point, by lemma~\ref{lem:unique_Sigma_n} \emph{iii)}, we get that
$[0,\theta)=S_{n}$ and the map $\s$ becomes continuous. Hence, as
$\s(S_{n})\subset S_{n}$ and $\s$ is a contracting map, there necessary exists a
fixed point $\bx_{n}\in S_{n}$ for $A=A_{n}^{\cal C}(d)$.

We now show that, when they exist, the fixed points $\bx_n$ are the only
possible invariant objects. On one hand note that the monotonicity of $f(x)$ and
the stability of $\bx\in(0,\theta)$ imply that $\s$ is monotonically increasing
and contracting in $S_n$.  On the other one we recall that, from
lemma~\ref{lem:unique_Sigma_n}, we use that at most two sets $S_{n-1}$ and $S_n$
or $S_n$ and $S_{n+1}$ can coexist. Hence, if $\s$ possesses a fixed point
$\bx_n\in S_n$, there exist only three possible situations:
\begin{enumerate}
\item $\bx_n\in(0,\theta)\subset S_n$
\item $[0,\Sigma_n)= S_{n-1}$ and
$\bx_n\in[\Sigma_n,\theta)= S_n$
\item $\bx_n\in(0,\Sigma_{n+1})\subset S_{n}$ and
$[\Sigma_{n+1},\theta)= S_{n+1}$
\end{enumerate}
Note that $\bx\in(0,\theta)$ implies that necessary $\bx_n\ne0$.\\
In the first case no other periodic points are possible because $\s$ is
monotonically increasing, contracting and continuous in all its domain.\\
In the second and third cases this becomes a direct consequence of the fact that
$\s$ has a negative gap at its discontinuities (illustrated in
fig.~\ref{fig:maps_bx2}); that is,
\begin{equation*}
\lim_{x\to(\Sigma_n)^-}\s(x)>\lim_{x\to(\Sigma_n)^+}\s(x),\,\forall
n\ge 0,
\end{equation*}
(see fig.~\ref{fig:boundary} for $n=3$). This is a direct consequence of
lemma~\ref{lem:lateral_values}.
\end{proof}

The next proposition provides us the bifurcation curves at which the fixed
points $\bx_n\in S_n$ given in proposition~\ref{pro:existence_of_fp} undergo
border collision bifurcations. That is, curves for which the fixed points
collide with the boundaries:
\begin{align*}
A\to A_n^\LL(d)&\Longrightarrow\bx_n\longrightarrow\left( \Sigma_n
\right)^-\\
A\to A_{n}^\R(d)&\Longrightarrow\bx_n\longrightarrow \left( \Sigma_{n-1}
\right)^+.
\end{align*}

\begin{prop}\label{pro:1T_po}
For any $0<d<1$ there exists a sequence
\begin{equation}
0<A_0(d)<A_1^\R(d)<A_1^\LL(d)<A_2^\R(d)<A_2^\LL(d)<\dots,
\label{eq:sequence_As}
\end{equation}
such that, for every $n>0$, the stroboscopic map $\s$~\eqref{eq:strobo_map}
possesses a unique fixed point $\bx_n\in S_n$ for
$A\in(A_n^\R(d),A_n^\LL(d))$. This fixed point undergoes a border collision
bifurcation at $A=A_n^\LL(d)$ and $A=A_n^\R(d)$.\\
The values $A_0(d)$, $A_n^\LL(d)$ and $A_n^\R(d)$ define smooth curves
fulfilling
\begin{align*}
&\lim_{d\to 0}A_0=\infty\\
&\lim_{d\to 0}A_n^{\R,\LL}(d)=\infty
\end{align*}
Hence these curves go through the origin of the parameter space $(d,1/A)$.
Moreover,
\begin{equation*}
\lim_{d\to 1}A_n^{\R}(d)=\lim_{d\to 1}A_n^{\LL}(d),\,n\ge1,
\label{eq:d1_limits}
\end{equation*}
and this limit becomes the value of $A$ for which
\begin{equation*}
\delta=\frac{T}{n},
\end{equation*}
where $\delta$ is defined in eq.~\eqref{eq:delta}.
\end{prop}

\begin{proof}
Let $\s_-$ and $\s_+$ be as in eqs.~\eqref{eq-defsmin} and~\eqref{eq-defsplus}:
\begin{align*}
\s_-&=\varphi(T(1-d);\theta;0)\\
\s_+&=\varphi(T(1-d);0;0).
\end{align*}
As given in lemma~\ref{lem:lateral_values}, if the stroboscopic map exhibits a
discontinuity, $\Sigma_n\in(0,\theta)$, then these values coincide with the
lateral images of $\Sigma_n$:
\begin{align*}
\s(\Sigma_n^-)&=\s_-\\
\s(\Sigma_n^+)&=\s_+.
\end{align*}
For $0<d<1$ and $n\ge1$, let $A_n^{\cal C}(d)$ be as in the proof of
prop.~\ref{pro:existence_of_fp}, i.e. the value of $A$ defined by the condition
$\Sigma_{n}=0$. By lemma~\ref{lem:monotonicity} and the implicit function
theorem $A_n^{\cal C}(d)$ is a smooth function of $d$.  We begin by discussing
the bifurcation sequence occurring for some $0< d<1$ and $0<A<A_1^{\cal C}(d)$.
The sequence of bifurcations that we now describe are illustrated in
fig.~\ref{fig:p-orbits} for a particular example.\\
First note that for $A$ sufficiently small, $[0, \theta)= S_0\cup S_1$ and
$\Sigma_1> \s_-=\s(\Sigma_1^-)$.  In this case the graph of $\s$ is as shown in
figure~\ref{fig:map3} (for $n=3$) and clearly the fixed point $\bar x_0$
contained in $S_0$ must exist. By lemma~\ref{lem:monotonicity}, when $A$ is
increased $\Sigma_1$ decreases with non-zero speed. Hence there exists a unique
value of $A=A_0$ such that $\Sigma_1=\s_-=\s(\Sigma_1^-)$ (see
figure~\ref{fig:map_T_L-bif} for $\Sigma_3$). In other words, the fixed point
$\bx_0$ undergoes a border collision bifurcation as it collides with the
boundary $\Sigma_1$ on its left.  Moreover, by lemma~\ref{lem:monotonicity} and
the implicit function theorem, the equation $\Sigma_1=\s_-$ defines a smooth
function $A_0(d)$.\\
For $A>A_0(d)$ the map $\s$ is as shown in figure~\ref{fig:boundary_map} and the
fixed point $\bar x_0$ no longer exists.  As $A$ is further increased $\Sigma_1$
crosses $\s_+$ for some value $A=A_1^\R$. For this value, the fixed point
$\bx_1$ undergoes a border collision bifurcation as it collides with the
boundary $\Sigma_1$ on its right (see figure~\ref{fig:map_T_R-bif} for $\bx_2$).
Similarly as above, the equation $\Sigma_1=\s_+$ defines a smooth function
$A=A^{\cal R}_1(d)$ and for $A>A^{\cal R}_1(d)$ the fixed point $\bar x_1$
exists.  Finally, for $A=A_1^{\cal C}(d)$ the map is continuous on the entire
$(0,\theta)$ with fixed point $\bar x_1$.

We now repeat the same argument to show that the same bifurcation sequence
occurs for $A_{n}^{\cal C}(d)<A\le A_{n+1}^{\cal C}(d)$, $n\ge 1$.  Indeed, for
$A>A_{n}^{\cal C}(d)$ but close to $A_{n}^{\cal C}(d)$ we have $\Sigma_{n}>
\s_-$.  In this case the graph of $\s$ is again as shown in
figure~\ref{fig:map3} and clearly the fixed point $\bar x_{n-1}$ contained in
$S_{n-1}$ must exist. By lemma~\ref{lem:monotonicity}, when $A$ is increased
$\Sigma_{n}$ decreases with non-zero speed. Hence there exists a unique value of
$A=A^{\cal L}_{n-1}$ such that  $\Sigma_{n}= \s_-=\ts(\Sigma_n^-)$ (see
figure~\ref{fig:map_T_L-bif} for $n=3$). For this value of $A$, the fixed point
$\bx_{n-1}$ undergoes a border collision as it collides with the boundary
$\Sigma_n$ on its left.  Moreover, by lemma~\ref{lem:monotonicity} and the
implicit function theorem, the equation $\Sigma_{n}= \s_-=\s(\Sigma_n^-)$
defines a smooth function $A^{\cal L}_{n}(d)$.\\
For $A>A^{\cal L}_{n-1}(d)$ the map $\s$ is as shown in
figure~\ref{fig:boundary_map} and the fixed point $\bar x_{n-1}$ no longer
exists. As $A$ is further increased, there exists some value $A=A_{n}^\R(d)$ for
which $\Sigma_{n}$ crosses $\s_+$, as in figure~\ref{fig:map_T_R-bif}. For this
value, the fixed point $\bx_n$ undergoes a border collision bifurcation as it
collides with the boundary $\Sigma_n$ on its right. Similarly as above the
equation $\Sigma_{n}= \s_+=\s(\Sigma_n^+)$ defines a smooth function $A=A^{\cal
R}_n(d)$ and for $A>A^{\cal R}_n(d)$ the fixed point $\bar x_{n}$ exists.
Finally, for $A=A_{n}^{\cal C}(d)$ the map is continuous on the entire
$(0,\theta)$ with fixed point $\bar x_n$.

Finally, note that due to the fact that the lateral values of $\Sigma_n$ by $\s$
($\s_-$ and $\s_{+}$) do not depend on $A$ and, by lemma~\ref{lem:monotonicity},
$\Sigma_n$ monotonically decreases with $A$, the bifurcations described above
are unique, as $\s(\Sigma_n^\pm)=\s_\pm$ can occur only once.

As will be shown in next section, the periodic orbits that exist in the
intervals of the form $(A_{n-1}^\LL(d),A_{n}^\R(d))$ are given by the period
adding strucutre.\\

We now discuss the limiting values of the curves $A_n^{\R,\LL}(d)$ when $d\to0$
and $d\to 1$.

Let $A_0(d)$ be the curve where the fixed point $\bx_0\in S_0$
undergoes a border collision. As discused above, this occurs when the fixed
point $\bx_0$ collides with the boundary:
\begin{align}
\varphi(dT;\bx_0;A_0)&=\theta\label{eq:bif_condition_A0}\\
\varphi(T-dT;\theta;0)&=\bx_0,\nonumber
\end{align}
where $\varphi(t;x;A)$ is the flow associated with system $\dot{x}=f(x)+A$.  Due
to the monotonicity of $f(x)$, $\varphi(dT;x;A)$ is a monotonotonically
increasing function of $A$ for $A\in [A_0,A_1^\R)$ for all $x\in[0,\theta)$.
Hence, if $d\to 0$ then $A_0\to\infty$ in order to keep
equation~\eqref{eq:bif_condition_A0} satisfied. Therefore the curve defined by
$A_0(d)$ goes through the origin of the parameter space $d\times 1/A$. The
sequence given in eq.~\eqref{eq:sequence_As} implies that all the other
bifurcation curves given by $A_n^{\R,\LL}(d)$ also go through this point.

We now focus on $d\to1$. Let $A\in(A_n^\R(d),A_n^\LL(d))$ and hence there exists
a fixed point of $\s$, $\bx_n\in S_n$. In order to see~\eqref{eq:d1_limits} we
will show that, if there exists such a fixed point, then necessary $\s(x)\to x$
when $d\to1$; that is, if a fixed point exists, then the stroboscopic map $\s$
tends to be the identity when $d\to1$. If that's the case, then, both
bifurcations tend to occur at the same time and hence
\begin{equation*}
A_n^\R(d)-A_n^\LL(d)\to 0,
\end{equation*}
and both limits exists, which proves eq.~\eqref{eq:d1_limits}.\\
To see that $\s$ tends to be the identity when $d\to1$ if $\bx_n\in S_n$ we
first recall that, for $d=1$ the system~\SYSTEMWR{} becomes the autonomous
system $\dot{x}=f(x)+A$ plus the reset condition~\eqref{eq:reset_condition}. As
$A\in(A_n^\R(d),A_n^\LL(d))$ is large enough to make the system exhibit spikes,
if $d=1$ there exists only a periodic orbit with period $\delta=\delta(A)$
defined in~\eqref{eq:delta}. Moreover, all points in $[0,\theta)$ belong to this
periodic orbit because of the reset. As a consequence, all initial conditions in
$[0,\theta)$ are fixed points of the time-$\delta$ return map and this map is
hence the identity.\\
We now write the conditions for the existence of the fixed point $\bx_n$ in
terms of the flow:
\begin{equation*}
\begin{aligned}
\varphi(t_n;\bx_n;A)&=\theta\\
\varphi(dT-t_n-(n-1)\delta;0;A)&=x'\\
\varphi(T-dT;x';0)&=\bx_n.
\end{aligned}
\end{equation*}
Clearly, when $d\to1$, $x'\to\bx_n$ and $A$ tends to be such that $T$ is a
multiple of  $\delta(A)$; indeed 
\begin{equation*}
n\delta(A)T\to T.
\end{equation*}
As a consequence, for this limiting value of $A$, the time-$T$ return map, $\s$,
also tends to be the identity in $S_n$.
\end{proof}

\subsection{Period adding structures}\label{sec:period_adding}
In this section we study the invariant objects (periodic orbits) located in the
regions in the parameter space $d\times 1/A$ nested in between fixed points of
the stroboscopic map (gray regions in fig.~\ref{fig:d-invA_generic_regions}).\\
As announced in~\S\ref{sec:results}, these are organized by period adding
structures. Using the background on this topic provided by
theorem~\ref{theo:adding}, this is a direct consequence of the following
proposition, which states that the stroboscopic map~\eqref{eq:strobo_map} can be
written as the normal form given in eq.~\eqref{eq:normal_form} for parameter
values between two consecutive regions where fixed points exist. This will be a
consequence of the fact that all fixed points simultaneously undergo border
collision bifurcations at the origin of the parameter space $d\times 1/A$.

\begin{prop}\label{pro:adding}
Let
\begin{equation}
\begin{array}{cccc}
\mu:&[0,1]&\longrightarrow&\RR^{2}\\
&\lambda&\longmapsto&\left( d(\lambda),B(\lambda) \right)
\end{array}
\label{eq:param_curve}
\end{equation}
be a $C^\infty$ parametrization of a curve in the parameter space $d\times 1/A$
such that $\mu(0)=(d_0,B_0)$ and $\mu(1)=(d_1,B_1)$ fulfilling
\begin{align*}
d_i&\in(0,1)\\
B_0&=\frac{1}{A_n^\LL(d_0)}\\
B_1&=\frac{1}{A_{n+1}^\R(d_1)},
\end{align*}
for some $n>0$, and
\begin{equation*}
B'(\lambda)<0,\,\lambda\in(0,1).
\end{equation*}
Let $\s_\lambda$ be the map obtained by applying the reparametrization $\mu$ to
$\s$.  Then, the bifurcation diagram associated with the map $\s_\lambda$ for
$\lambda\in[0,1]$ follows a period adding as described for the
map~\eqref{eq:normal_form}.
\end{prop}

\begin{proof}
We show that the map $\s$ can be written in the form of
map~\eqref{eq:normal_form} fulfilling the conditions of
theorem~\ref{theo:adding}.

After performing the reparametrization $\lambda \longmapsto(d,A)$ given
by the\linebreak curve~\eqref{eq:param_curve}, we have that
$\Sigma_n(\lambda)\in[0,\theta)$ $\forall \lambda \in[0,1]$. Hence, the change
of variables
\begin{equation*}
\tx\longmapsto x-\Sigma_n(\lambda),
\end{equation*}
is well defined and makes the map
$\ts_\lambda(\tx):=\s(\tx+\Sigma_n(\lambda))-\Sigma_n(\lambda)$ undergo a
discontinuity at $\tx=0$, independently of $\lambda$. The stroboscopic map
$\ts_\lambda$ is continuous, increasing and contracting in
$[0,\theta-\Sigma_n(\lambda))$ and $[-\Sigma_n(\lambda),0)$ for
$\lambda\in[0,1]$. Hence, the map $\ts_\lambda(\tx)$ is of
type~\eqref{eq:normal_form} and
satisfies~\eqref{eq:increasing_contracting_condition} in its domain, which is
enough as argued in remark~\ref{rem:local_contractiveness}.

As argued in the proof of proposition~\ref{pro:1T_po}, the map
$\ts_\lambda(\tx)$ undergoes a negative gap at $\tx=0$ for $\lambda\in[0,1]$:
\begin{equation}
\ts_\lambda(0^-)>\ts_\lambda(0^+).
\label{eq:negative_gap}
\end{equation}
For $\lambda=0$ and $\lambda=1$, the fixed points $\bx_{n-1}-\Sigma_n(0)$ and
$\bx_n-\Sigma_n(1)$ of the map $\ts_\lambda(\tx)$ undergo border collision
bifurcations. This implies
\begin{align*}
&\ts_0(0^-)=0,\\
&\ts_1(0^+)=0.
\end{align*}
Hence, condition C.3 of theorem~\ref{theo:adding} is satisfied. Recalling the negative
gap~\eqref{eq:negative_gap}, the map $\ts_\lambda(\tx)$ fulfills
\begin{align*}
&\ts_\lambda(0^+)<0,\,\forall \lambda\in(0,1),\\
&\ts_\lambda(0^-)>0,\,\forall \lambda\in(0,1),
\end{align*}
and hence C.1 in theorem~\ref{theo:adding} is satisfied too.\\
From the fact that $B'(\lambda)<0$, the reparametrization is such that the
parameter $A$ monotonically decreases with $\lambda$. From
lemma~\ref{lem:monotonicity} we get that $\Sigma_n(\lambda)$ also decreases
monotonically with $\lambda$ and hence C.2 in theorem~\ref{theo:adding} is also
satisfied.

Finally, we can apply theorem~\ref{theo:adding} for the normal form
map~\eqref{eq:normal_form} to obtain that the bifurcation scenario exhibited by
the map $\ts_\lambda$ (an thus by the original map $\s_\lambda$) when varying
$\lambda$ is given by the period adding structure.
\end{proof}

By combining the previous result and the bifurcation scenario described for the
normal form map~\eqref{eq:normal_form} we get that the regions in parameter
space between curves of the form $(d,1/A_n^\LL(d))$ and $(d,1/A_{n+1}^\R(d))$
are covered by an infinite number of periodic orbits with arbitrarily high
period following the period adding structure described in~\S\ref{sec:results}.
In addition, by combining it with lemma~\ref{pro:1T_po} we get the full
description of the bifurcation scenario for system~\SYSTEMWR{} in the parameter
space $(d,1/A)$, as described in~\S\ref{sec:results}.\\

\subsection{Symbolic sequences, firing number and firing rate}\label{sec:symbolic}
We now introduce the symbolic dynamics described in~\S\ref{sec:results}. Let
$(x_1,\dots,x_n)$ be a $n$-periodic orbit located in the region in parameter
space between the curves $(d,1/A_n^\LL)$ and $(d,1/A_{n+1}^\R)$. We then assign
to this orbit the symbolic sequence given by the encoding
\begin{equation*}
\begin{aligned}
&x_i\to \LL&&\text{if }x_i<\Sigma_n\\
&x_i\to\R&&\text{if }x_i\ge\Sigma_n.
\end{aligned}
\end{equation*}
Note that $x_i<0$ and $x_i\ge0$ correspond to $x_i\in S_n$ and $x_i\in
S_{n+1}$, respectively. Hence, recalling the definition of the sets $S_i$
provided in eq.~\eqref{eq:sets}, when introducing the
encoding~\eqref{eq:encoding}, the symbols $\LL$ and $\R$ have to be interpreted
as the system spiking $n$ and $n+1$ times for one iteration of the stroboscopic
map, respectively. This determines the so-called \emph{firing number},
introduced in~\cite{KeeHopRin81} as follows. 
\begin{definition}\label{def:firing_number}
Let $s$ be the total number of spikes performed by a $p$-periodic orbit of the
stroboscopic map $\s(x)$ when iterated $p$ times, $n,p\in\mathbb{N}$; then we
define the firing-number as
\begin{equation}
\eta=\frac{s}{p}.
\label{eq:average_spkes_pT}
\end{equation}
\end{definition}
Provided that these periodic orbits are attractive, this number becomes the
asymptotic average number of spikes per iteration of the stroboscopic map.\\
Note that this quantity can be computed from the symbolic sequence. Let
$\sigma$ be a $p$-periodic orbit for parameter values located between the curves
$1/A=1/A_n^\LL(d)$ and $1/A=1/A_{n+1}^\R$ whose symbolic sequence contains
$m$ $\R's$ and $k$ $\LL's$ ($k=p-m$). Then, the firing number becomes
\begin{equation*}
\eta=\frac{nk+(n+1)m}{p}=n+\frac{m}{p}.
\end{equation*}
As explained in~\S\ref{sec:results}, the rotation numbers associated with the
periodic orbits in the period adding structure can be computed dividing the
number of $\R's$ symbols in its symbolic sequence by its total period,
\begin{equation*}
\rho=\frac{m}{p}.
\end{equation*}
Hence, the firing number can be related with the rotation number as
\begin{equation}
\eta=n+\rho,
\label{eq:firing-rotation_number}
\end{equation}
where $n$ is such that the periodic orbit steps on each side of the boundary
$\Sigma_n$.\\
As explained also in~\S\ref{sec:results}, the rotation number is organized by
the Farey tree structure shown in fig.~\ref{fig:farey_tree} when parameters are
varied from the bifurcation curve $1/A=1/A_n^\LL(d)$ towards $1/A=1/A_{n+1}^\R$
along curves as the one given in proposition~\ref{pro:adding}.  Hence, it
follows a devil's staircase from $0$ to $1$. As a consequence, the firing number
follows a devil's staircase from $n$ to $n+1$ when parameters $d$ and $A$ are
varied as mentioned.

Taking into account that the stroboscopic map consists of flowing
system~\SYSTEMWR{} for a time $T$, the firing number allows one to compute the
\emph{firing rate} associated with the corresponding periodic orbit as
\begin{equation*}
r=\frac{\eta}{T}.
\end{equation*}
This becomes the asymptotic number of spikes per unit time.\\
Note that, provided that these periodic orbits are attracting, $r$ becomes the
asymptotic firing rate defined in eq.~\eqref{eq:defirate},
\begin{equation*}
r=\lim_{\tau\to\infty}\frac{\#(\mbox{spikes performed by } \phi(t;x_0)
\mbox{ for } t\in[0,\tau])}{\tau},
\end{equation*}
which does not depend on $x_0\in[0,\theta)$.\\
Using relation given in eq.~\eqref{eq:firing-rotation_number}, the firing rate
follows also a devil's staircase from $n/T$ and $(n+1)/T$ when the parameter
values are varied through curves as the one in prop.~\ref{pro:adding} between
the bifurcation curves $1/A=1/A_n^\LL(d)$ and $1/A=1/A_{n+1}^\R(d)$,
respectively.

\section{Examples}\label{sec:examples}
In this section we illustrate the results presented so far in two different
examples. We consider a system of the type
\begin{equation*}
\dot{x}=f(x)+I(t),
\end{equation*}
with $I(t)$ a $T$-periodic input as defined in eq~\eqref{eq:forcing} and three
different types of functions for $f(x)$ satisfying conditions~\conds{}.
The first example consists of a linear system, leading to the so-called
\emph{linear integrate-and-fire} system (LIF):
\begin{equation}
f_1(x)=ax+b.
\label{eq:linear}
\end{equation}
Taking the threshold $\theta=1$, conditions~\conds{} are satisfied if $a<0$ and
$-b/a\in(0,1)$. In figure~\ref{fig:d-invA_T1d9} we show the existence of
periodic orbits in the parameter space $d\times 1/A$, which is as predicted.  In
fig.~\ref{fig:d-invA_1dscann_periods} we show the periods of the periodic orbits
found along the line drawn in figure~\ref{fig:d-invA_T1d9}, although the
bifurcation scenario is topologically the same for any other curve transversally
crossing the colored regions. As one can see in
fig.~\ref{fig:d-invA_1dscann_periods}, along such a curve there exist regions
where one finds only $T$-periodic orbits (fixed points of the stroboscopic map,
period $1$). These are the black regions in figure~\ref{fig:d-invA_T1d9}, for
example containing the points as $B$, $C$, $D$ and $E$. The periodic orbits for
the parameter values for $B$ and $C$ are shown in figs.~\ref{fig:poB}
and~\ref{fig:poC}, respectively.  Note that their associated firing number
defined in eq.~\eqref{eq:firing-rotation_number} (average number of spikes per
iteration of the stroboscopic map) is $0$ and $1$, respectively.\\
As given by propositions~\ref{pro:existence_of_fp} and~\ref{pro:1T_po}, there
exist an infinite number of regions, accumulating at the horizontal axis
$1/A=0$, for which one finds $T$-periodic orbits with arbitrarily large integer
firing number. This can be seen in figure~\ref{fig:d-invA_1dscann_FN} for the
points $D$ and $E$ and successive regions containing $T$-periodic orbits
(period one figure~\ref{fig:d-invA_1dscann_periods}).
\begin{figure}
\begin{center}
\includegraphics[width=0.5\textwidth,angle=-90]
{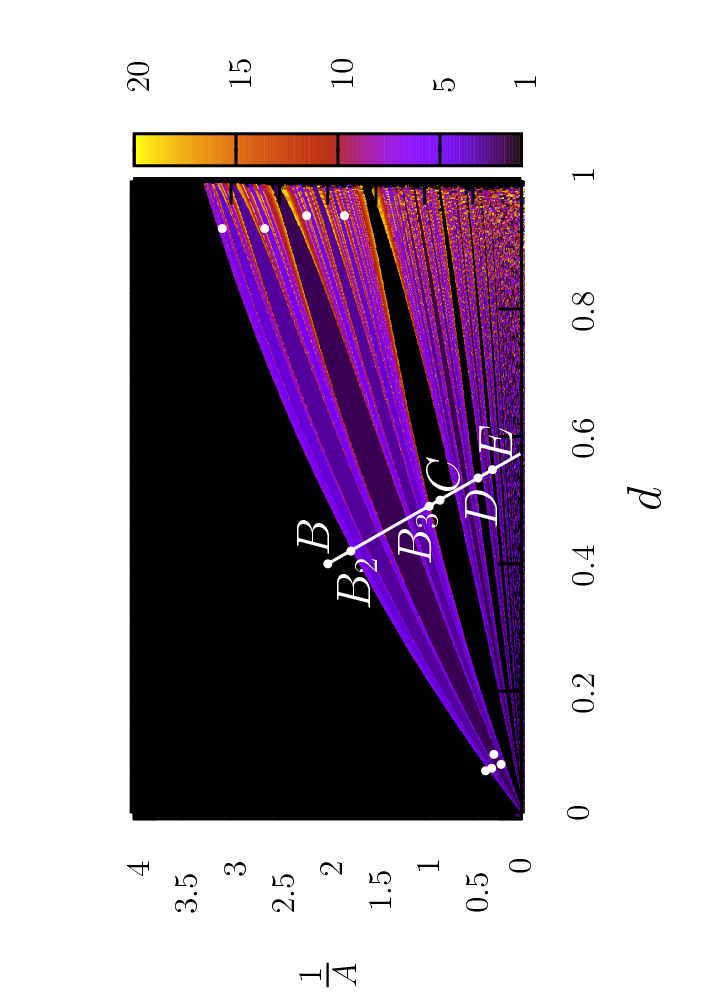}
\end{center}
\caption{Bifurcation diagram for the map $\s$ for $T=1.9$ associated with
system~\SYSTEMWRLinear{} for the linear system~\eqref{eq:linear}. The colors
refer to the periods of the periodic orbits found by simulating the system.
Periodic orbits are shown only up to period $20$ for clarity reasons. The
$T$-periodic orbits (fixed points of $\s$) for parameter values corresponding to
points $B$ and $C$, as well as their bifurcations ($B_2$ and $B_3$), are shown
in fig.~\ref{fig:p-orbits}. The non-labeled eight points correspond to
$5T$-periodic orbits with four different symbolic sequences. The ones for the
four points with small values of $d$ are shown in
fig.~\ref{fig:period5_orbits_smalld}, and the ones with large values of $d$ in
fig.~\ref{fig:period5_orbits_larged}.}
\label{fig:d-invA_T1d9}
\end{figure}
\begin{figure}
\begin{center}
\begin{picture}(1,0.5)
\put(0,0.35){
\subfigure[\label{fig:d-invA_1dscann_periods}]
{\includegraphics[angle=-90,width=0.5\textwidth]
{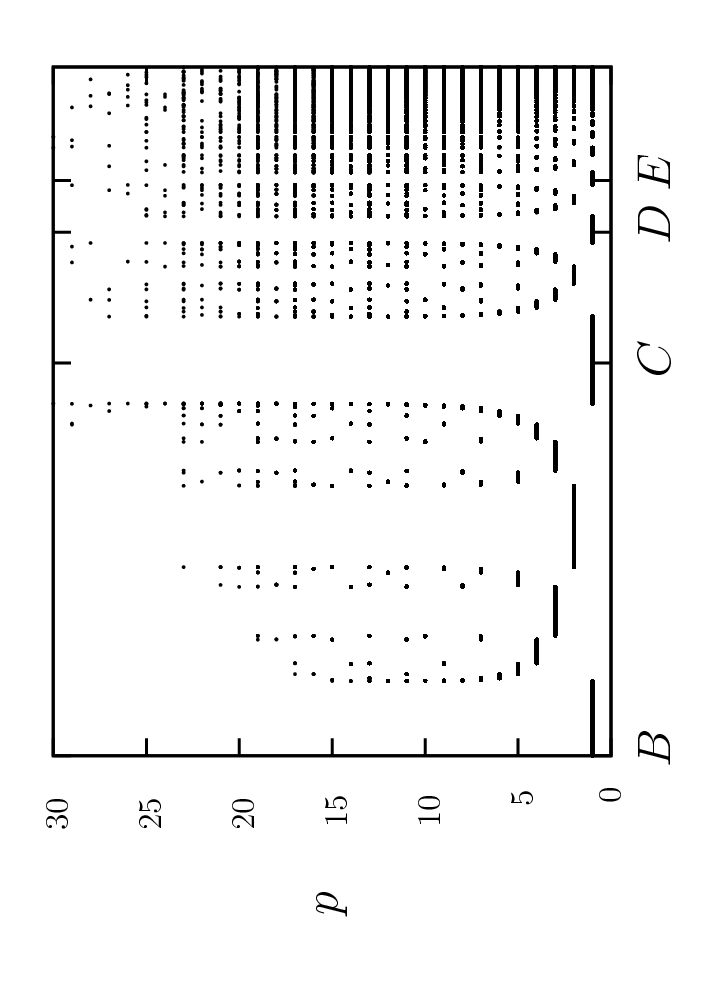}}
}
\put(0.5,0.35){
\subfigure[\label{fig:d-invA_1dscann_FN}]
{\includegraphics[angle=-90,width=0.5\textwidth]
{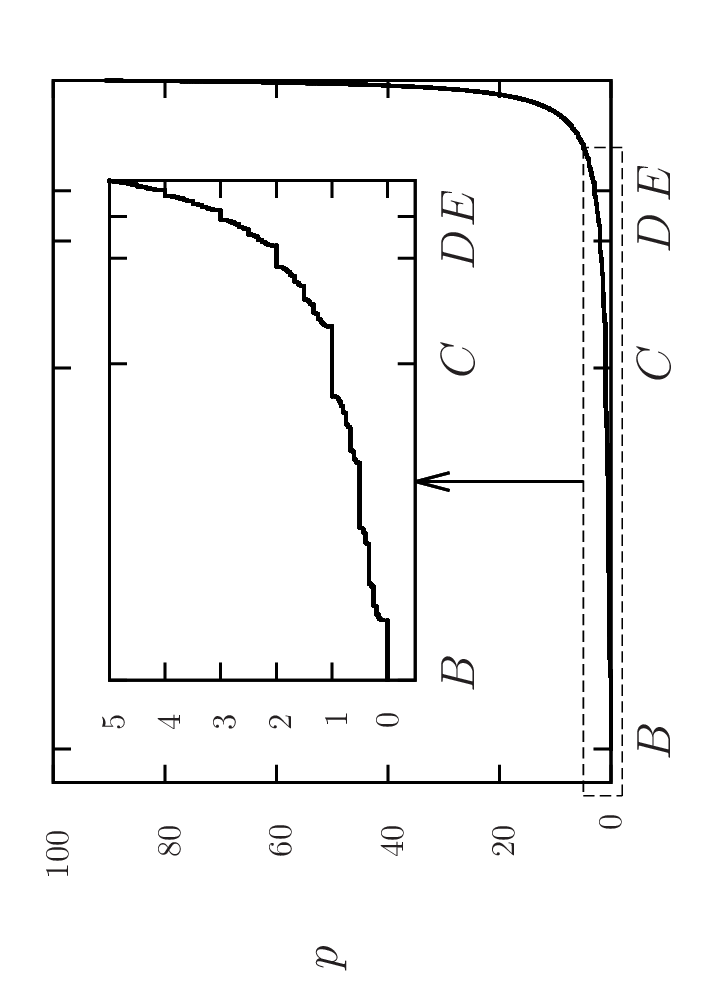}}
}
\end{picture}
\end{center}
\caption{(a) periods of the periodic orbits found along a line as the one shown
in fig.~\ref{fig:d-invA_T1d9}. (b) firing number associated with such periodic
orbits.}
\label{fig:1dscanns_linear}
\end{figure}

The two $T$-periodic orbits found in $B$ and $C$ undergo border collision
bifurcations at the points $B_2$ and $B_3$ labeled in
fig.~\ref{fig:d-invA_T1d9}. The periodic orbits at the moment of the bifurcation
are shown in figures~\ref{fig:poB2} and~\ref{fig:poB4}, respectively. This
occurs similarly for all other $T$-periodic orbits with higher firing number
(see proposition~\ref{pro:1T_po} for more details). In between points $B$ and
$C$, $C$ and $D$, $D$ and $E$ etc., a period adding bifurcation occurs; that is,
there exist an infinite number of periodic orbits whose periods, symbolic
sequences and rotation numbers are organized by the Farey tree structure shown
in fig.~\ref{fig:farey_tree}. After applying the symbolic
encoding~\eqref{eq:encoding}, the symbolic sequences and periods are obtained by
successive concatenation and addition as shown in fig.~\ref{fig:farey_tree} and
explained in section~\ref{sec:results}. In
figures~\ref{fig:period5_orbits_smalld} and~\ref{fig:period5_orbits_larged} we
show, for different parameter values, the four periodic orbits with period $5$
located in the adding structure between points $B$ and $C$, which correspond to
the symbolic sequences $\LL^4\R$, $\LL^2\R\LL\R$, $\LL\R\LL\R^2$ and $\LL\R^4$
(see the Farey tree in fig.~\ref{fig:farey_tree}). The parameter values used are
given by the points marked in fig.~\ref{fig:d-invA_T1d9}, for large
and small $d$. Note how the $\LL$ symbol corresponds to an iteration of the
stroboscopic map without exhibiting any spike, whereas for $\R$  one spike
occurs.

\begin{figure}
\begin{center}
\begin{picture}(1,0.8)
\put(0,0.8){
\subfigure[\label{fig:poB}]{\includegraphics[angle=-90,width=0.5\textwidth]
{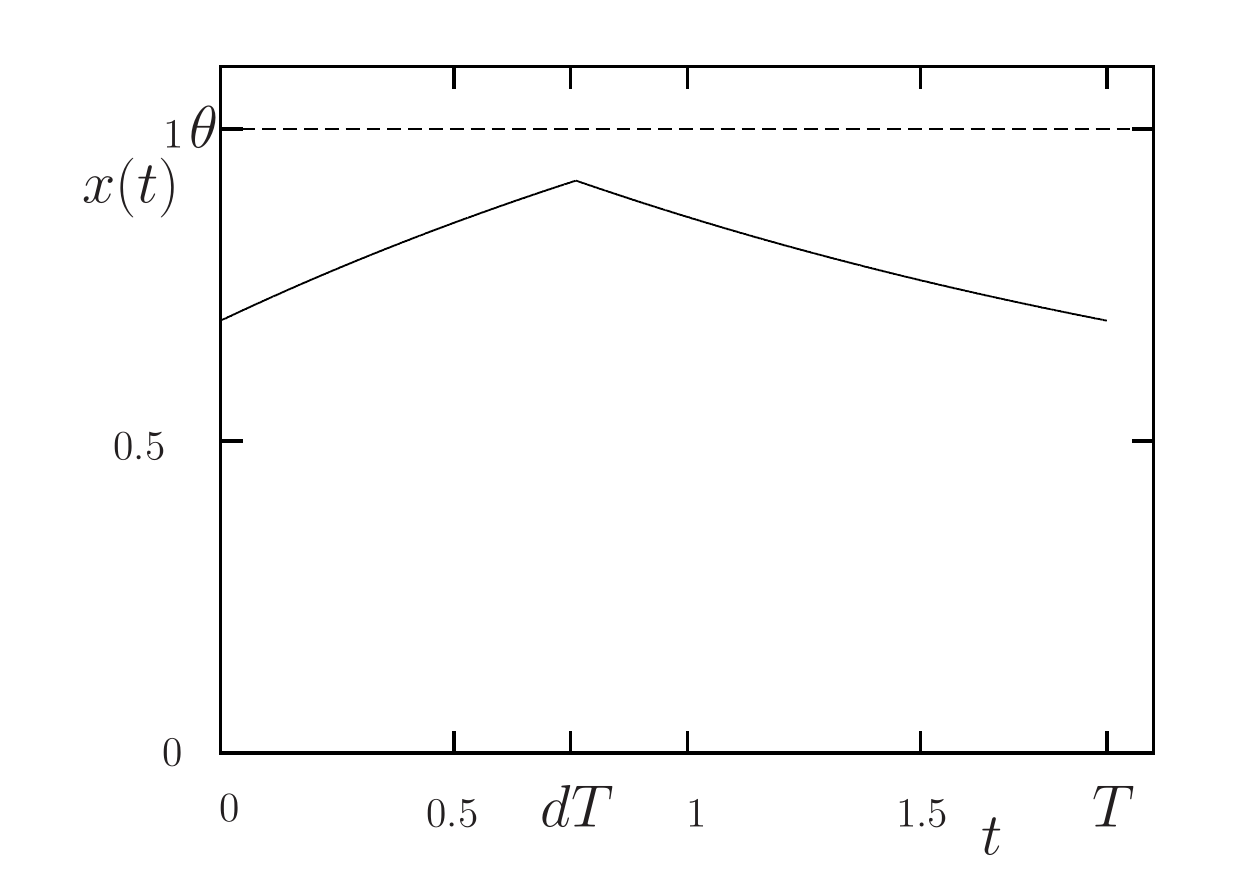}}
}
\put(0.5,0.8){
\subfigure[\label{fig:poB2}]{\includegraphics[angle=-90,width=0.5\textwidth]
{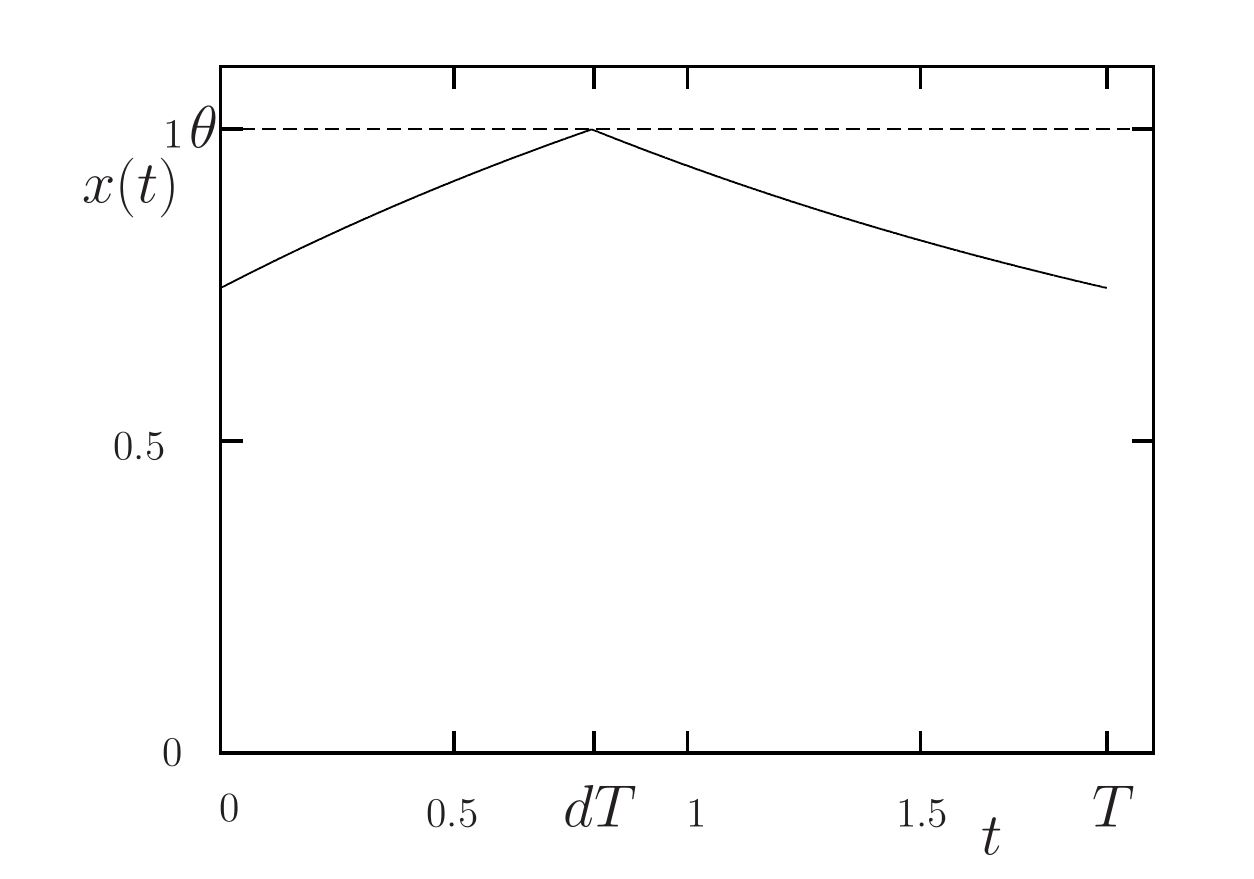}}
}
\put(0,0.4){
\subfigure[\label{fig:poB4}]
{\includegraphics[angle=-90,width=0.5\textwidth]
{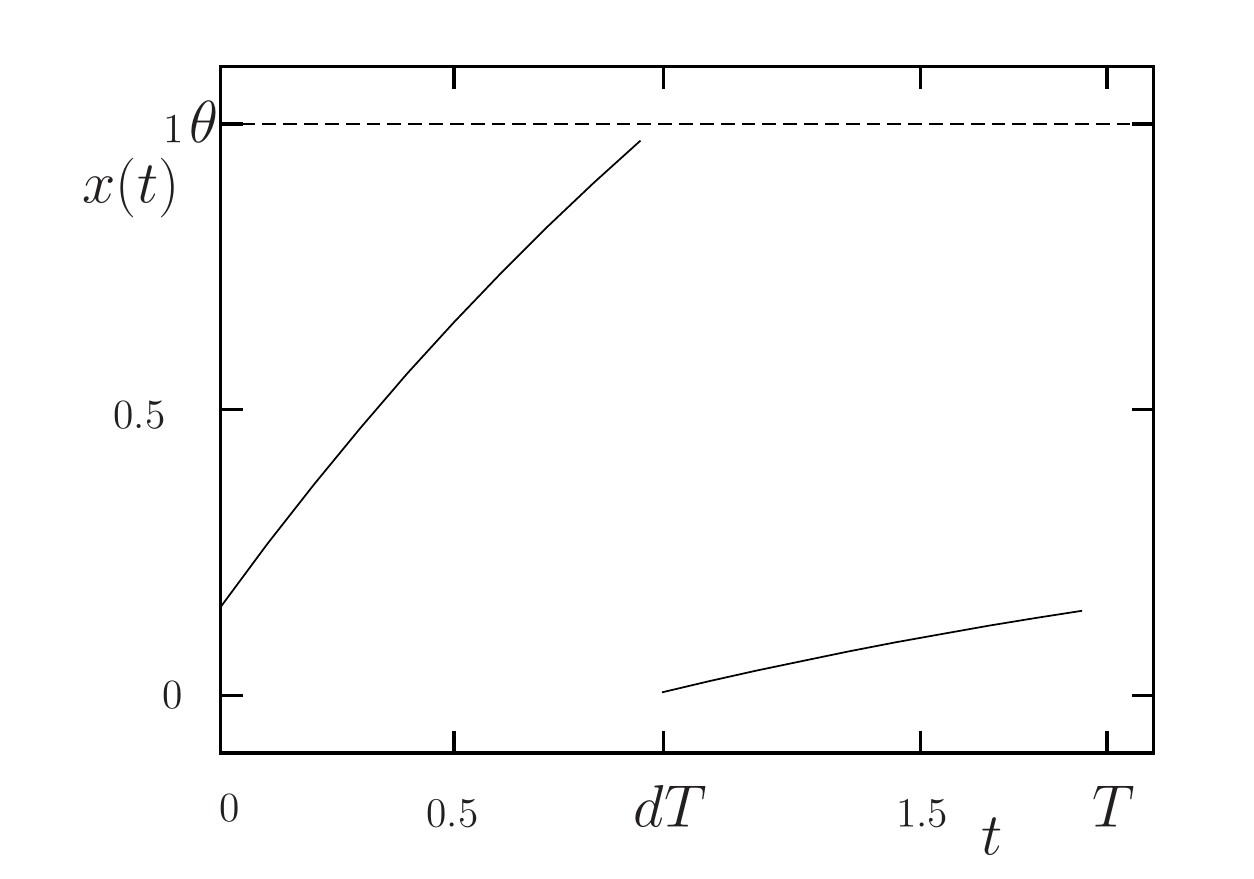}}
}
\put(0.5,0.4){
\subfigure[\label{fig:poC}]{
\includegraphics[angle=-90,width=0.5\textwidth]
{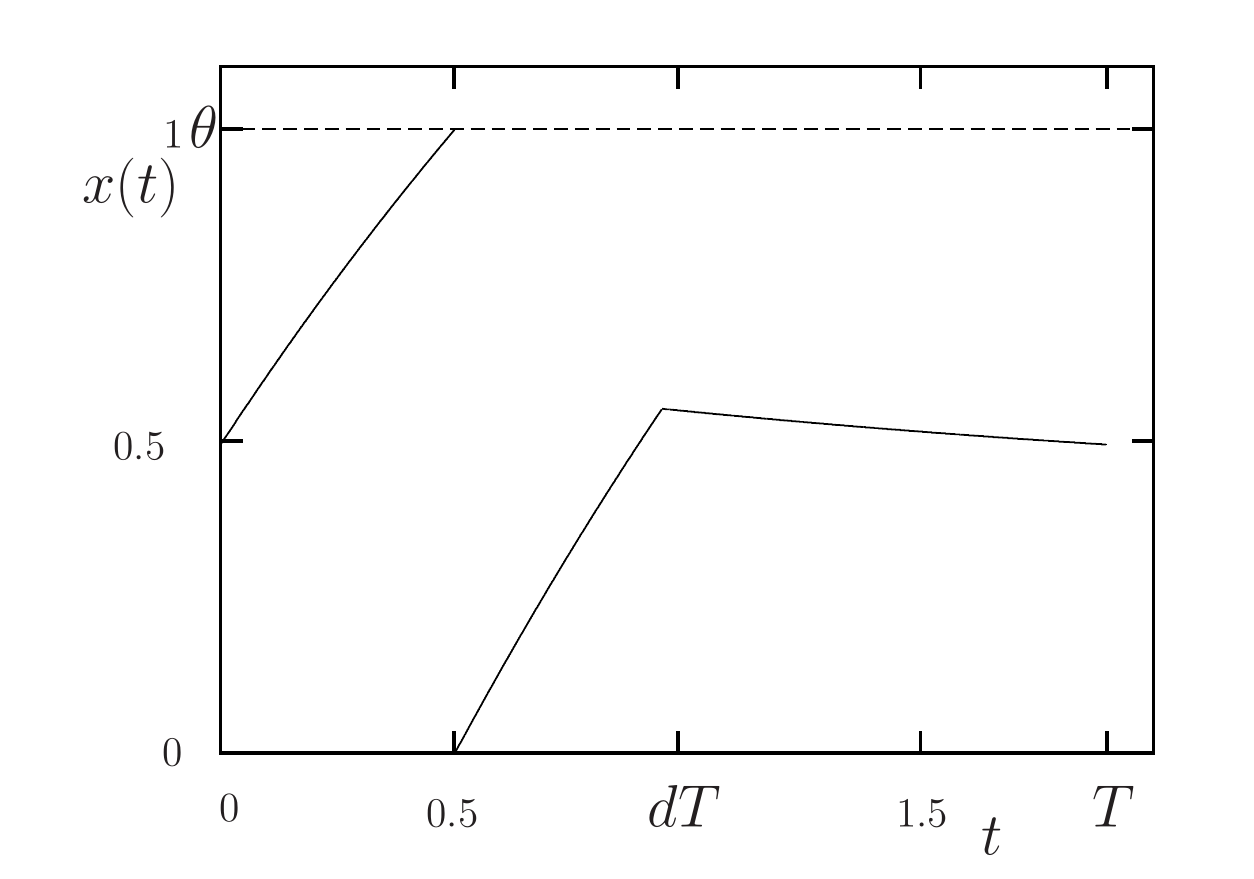}}
}
\end{picture}
\end{center}
\caption{Periodic orbits for the labeled points in fig.~\ref{fig:d-invA_T1d9}.
(a) periodic orbit with no spikes (fixed point $\bx_0$), (b) left bifurcation of
the fixed point, $\bx_0=\Sigma_1$, (c) right bifurcation of the fixed point
$\bx_1=\Sigma_1$, (d) periodic orbit spiking once.}
\label{fig:p-orbits}
\end{figure}

\begin{figure}
\begin{picture}(1,0.8)
\put(0,0.8){
\subfigure[\label{fig:L4R_smalld}]{\includegraphics[angle=-90,width=0.5\textwidth]
{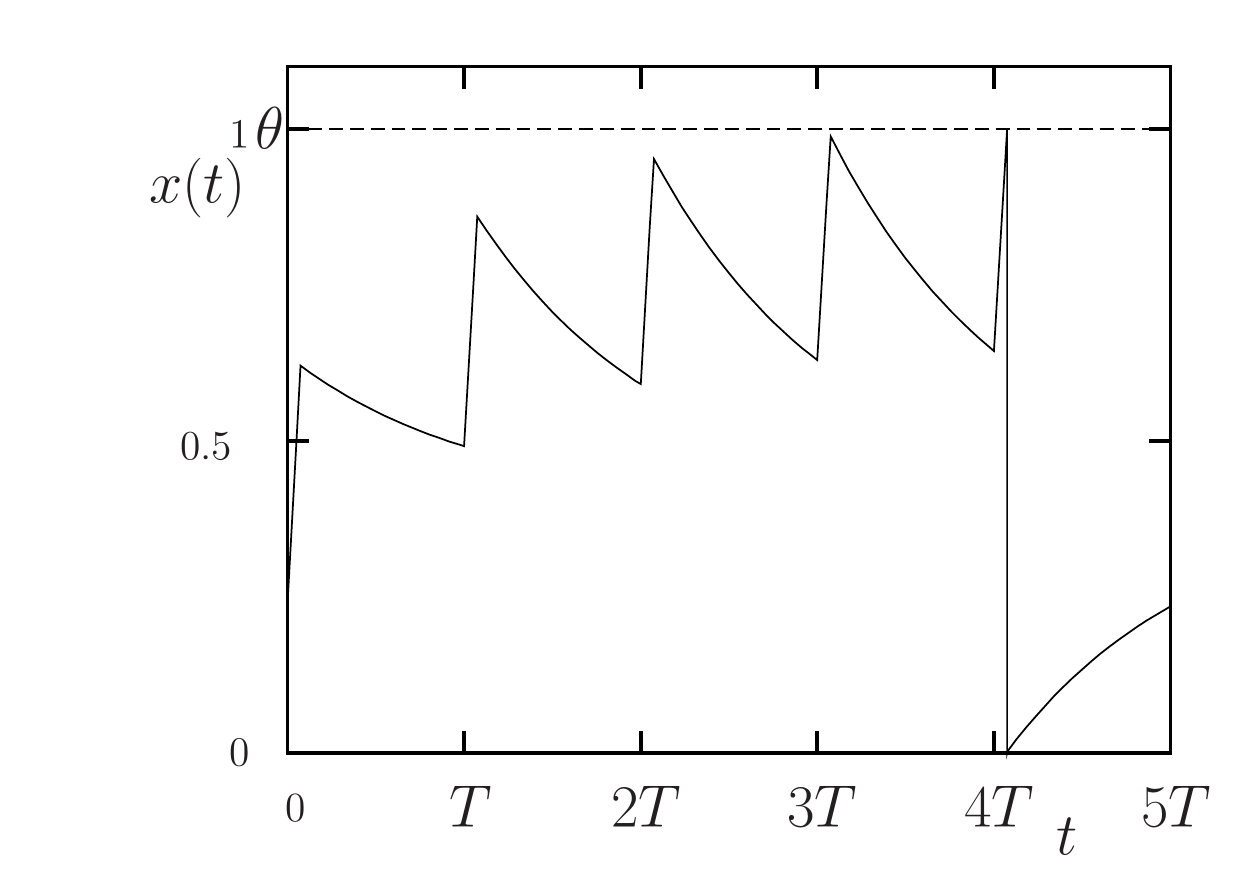}}
}
\put(0.5,0.8){
\subfigure[\label{fig:L2RLR_smalld}]{\includegraphics[angle=-90,width=0.5\textwidth]
{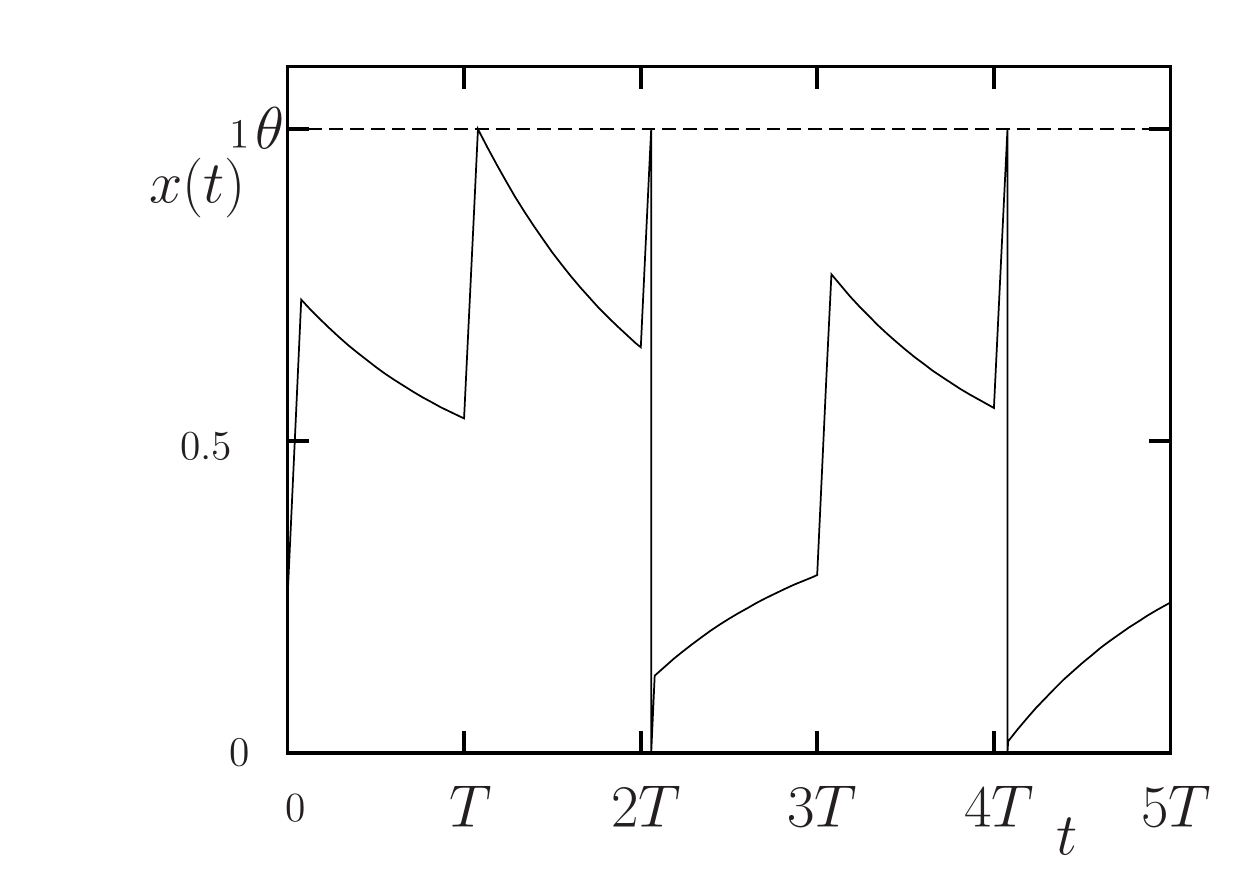}}
}
\put(0,0.4){
\subfigure[\label{fig:LRLR2_smalld}]{\includegraphics[angle=-90,width=0.5\textwidth]
{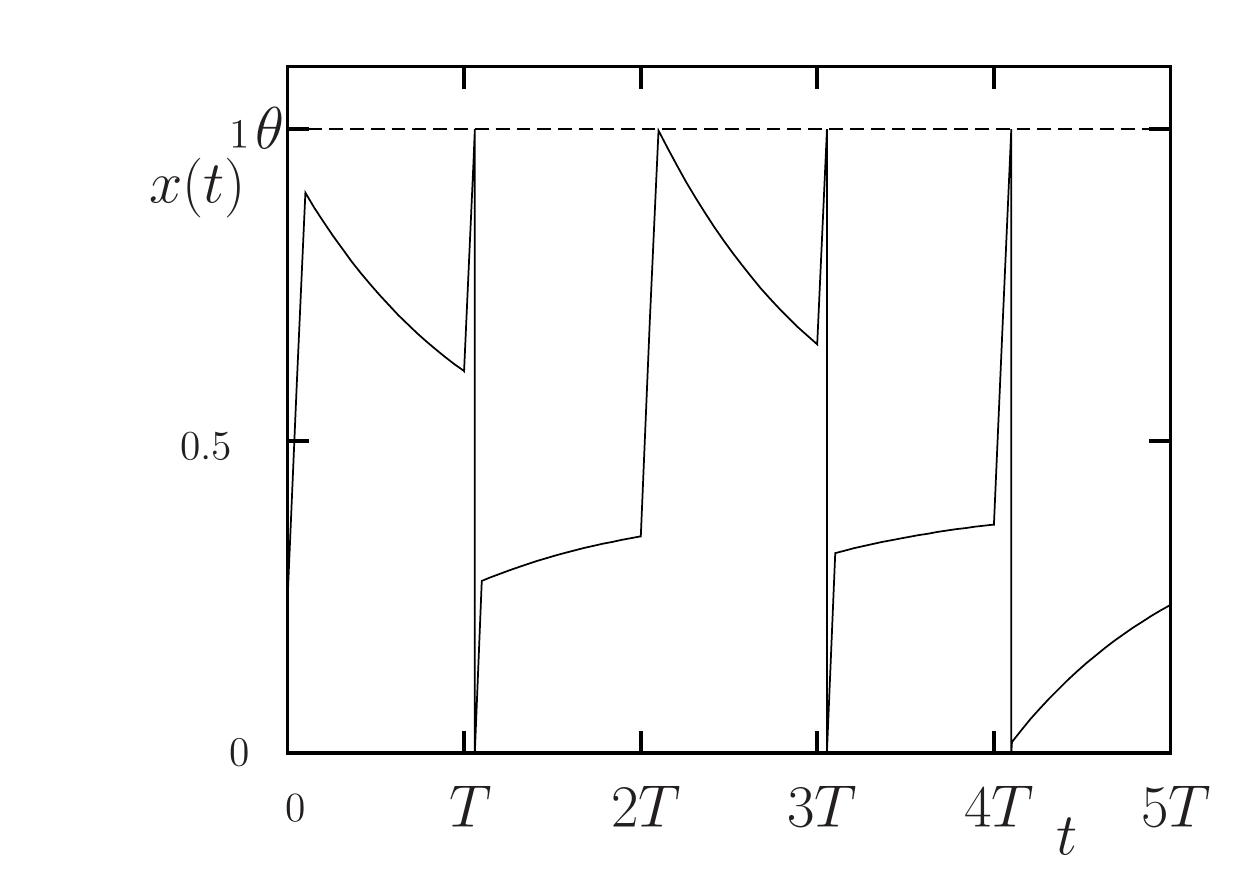}}
}
\put(0.5,0.4){
\subfigure[\label{fig:LR4_smalld}]{\includegraphics[angle=-90,width=0.5\textwidth]
{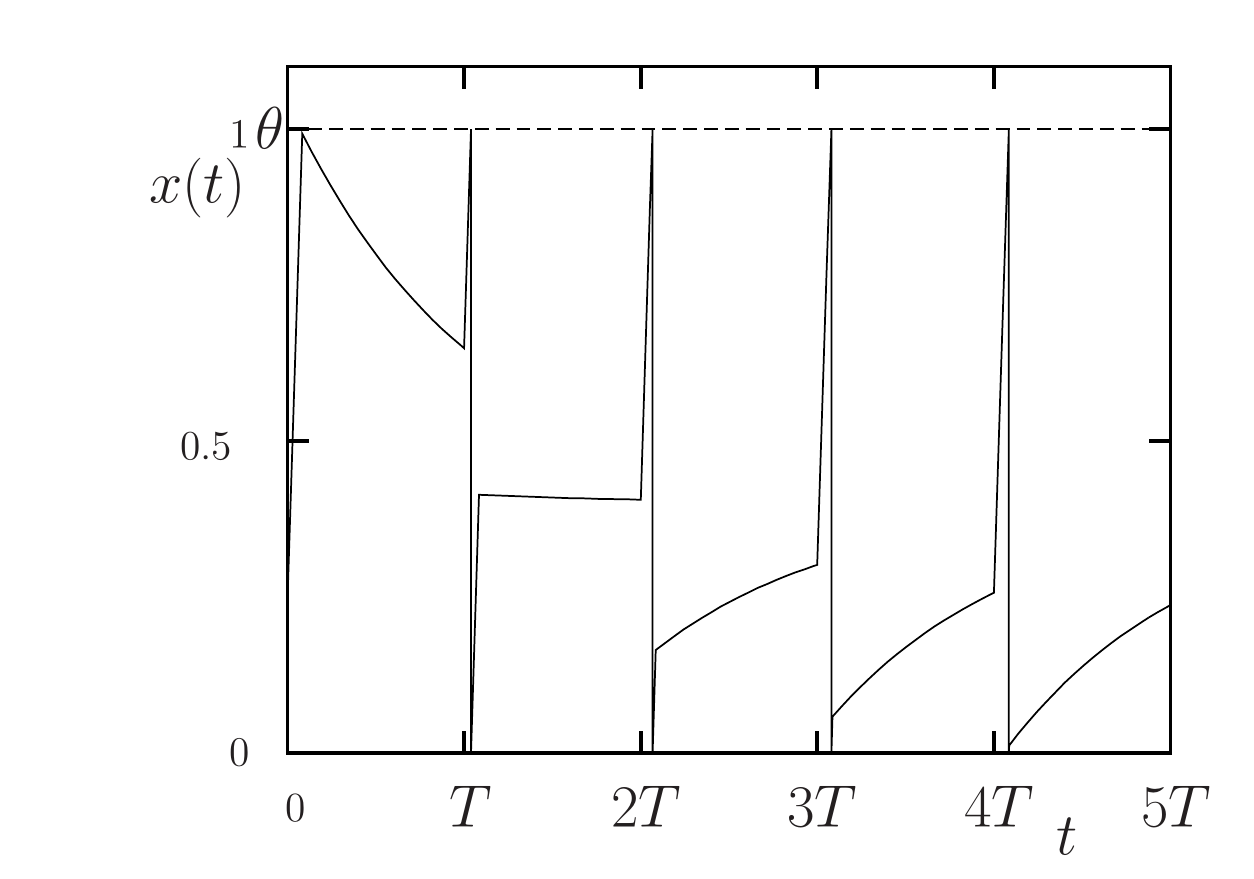}}
}
\end{picture}
\caption{Period-$5$ periodic orbits with different symbolic sequences and firing
number: $\LL^4\R$ $\eta=1/5$(a), $\LL^2\R\LL\R$ $\eta=2/5$ (b), $\LL\R\LL\R^2$
$\eta=3/5$ (c) and $\LL\R^4$ $\eta=4/5$ (d).  Parameter values for which these
periodic orbits exist are marked with points in figure~\ref{fig:d-invA_T1d9},
and correspond to small values of the duty cycle $d$.}
\label{fig:period5_orbits_smalld}
\end{figure}

\begin{figure}
\begin{picture}(1,0.8)
\put(0,0.8){
\subfigure[\label{fig:L4R_larged}]{\includegraphics[angle=-90,width=0.5\textwidth]
{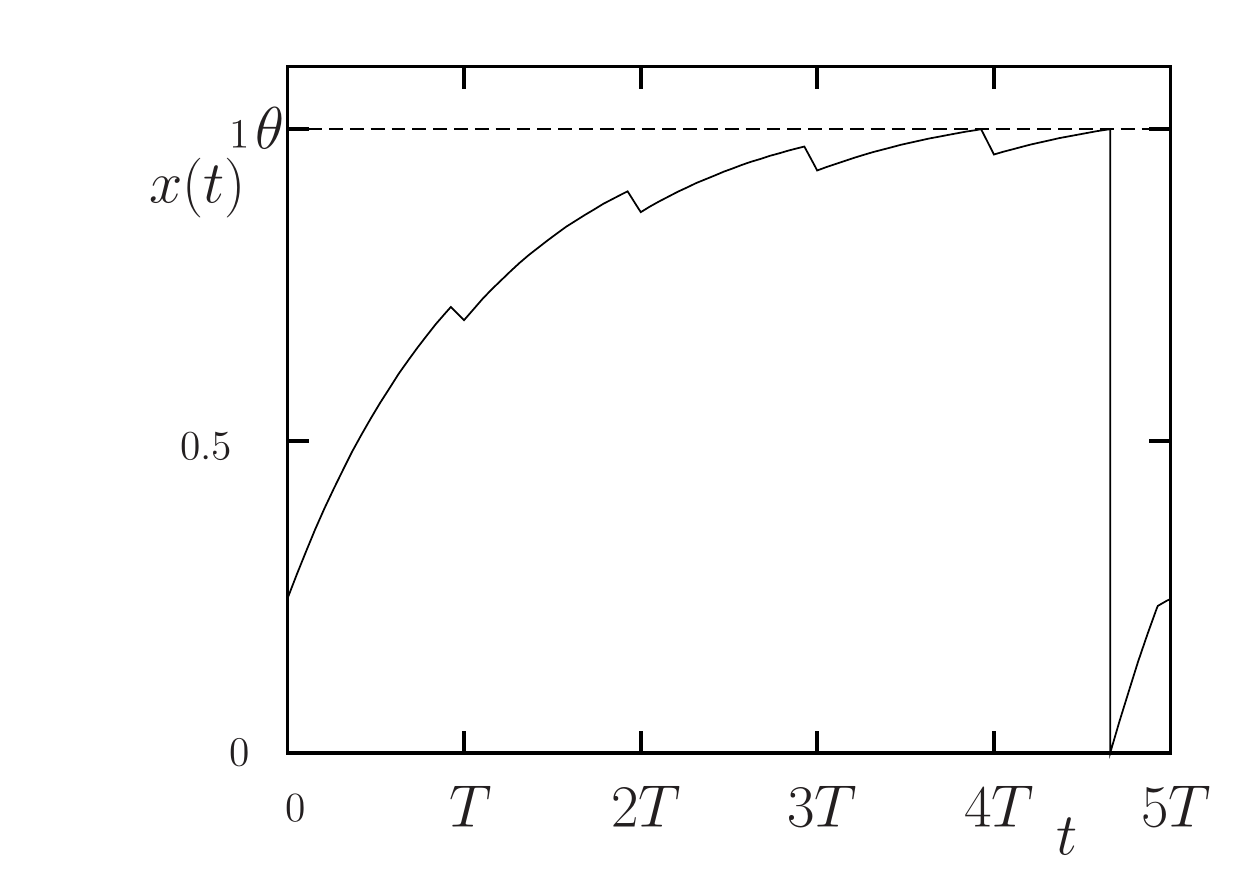}}
}
\put(0.5,0.8){
\subfigure[\label{fig:L2RLR_larged}]{\includegraphics[angle=-90,width=0.5\textwidth]
{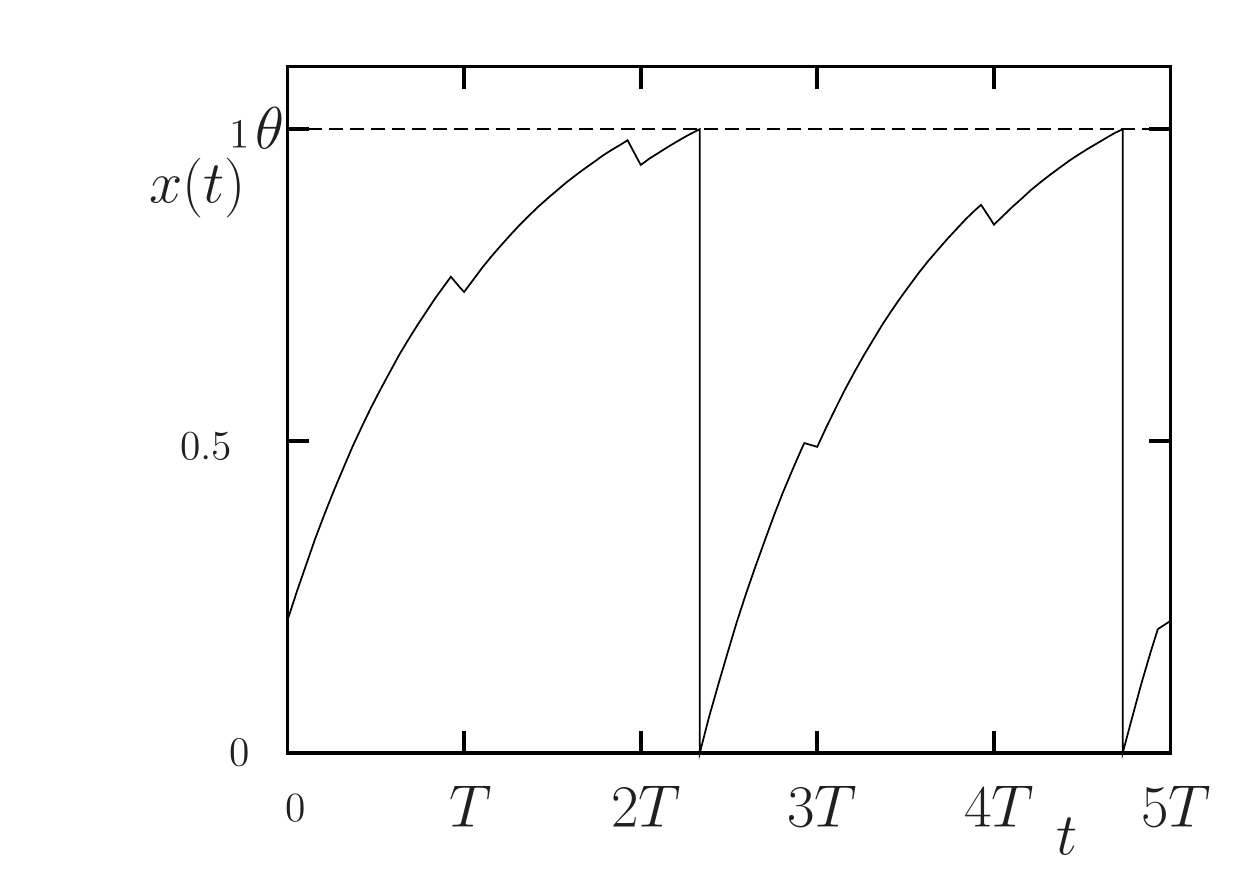}}
}
\put(0,0.4){
\subfigure[\label{fig:LRLR2_larged}]{\includegraphics[angle=-90,width=0.5\textwidth]
{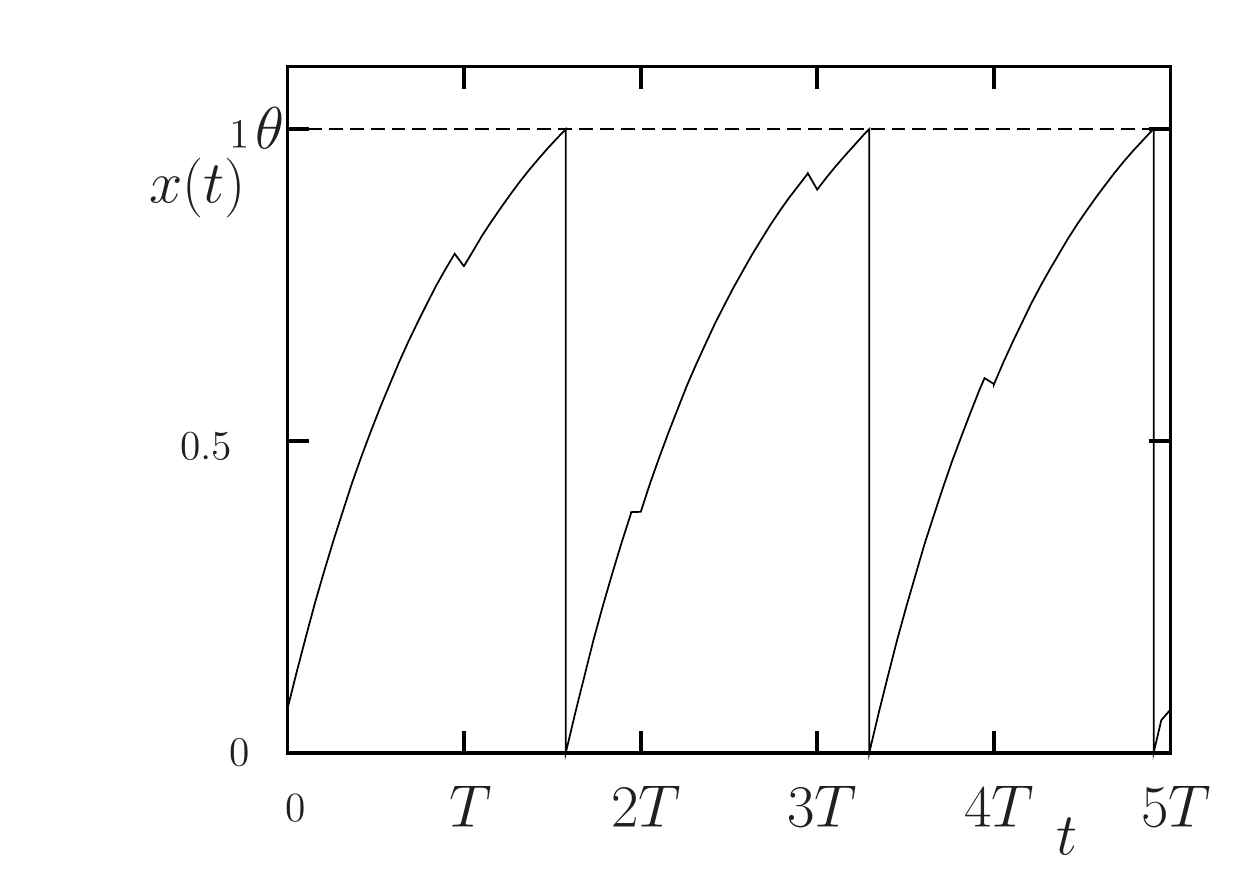}}
}
\put(0.5,0.4){
\subfigure[\label{fig:LR4_larged}]{\includegraphics[angle=-90,width=0.5\textwidth]
{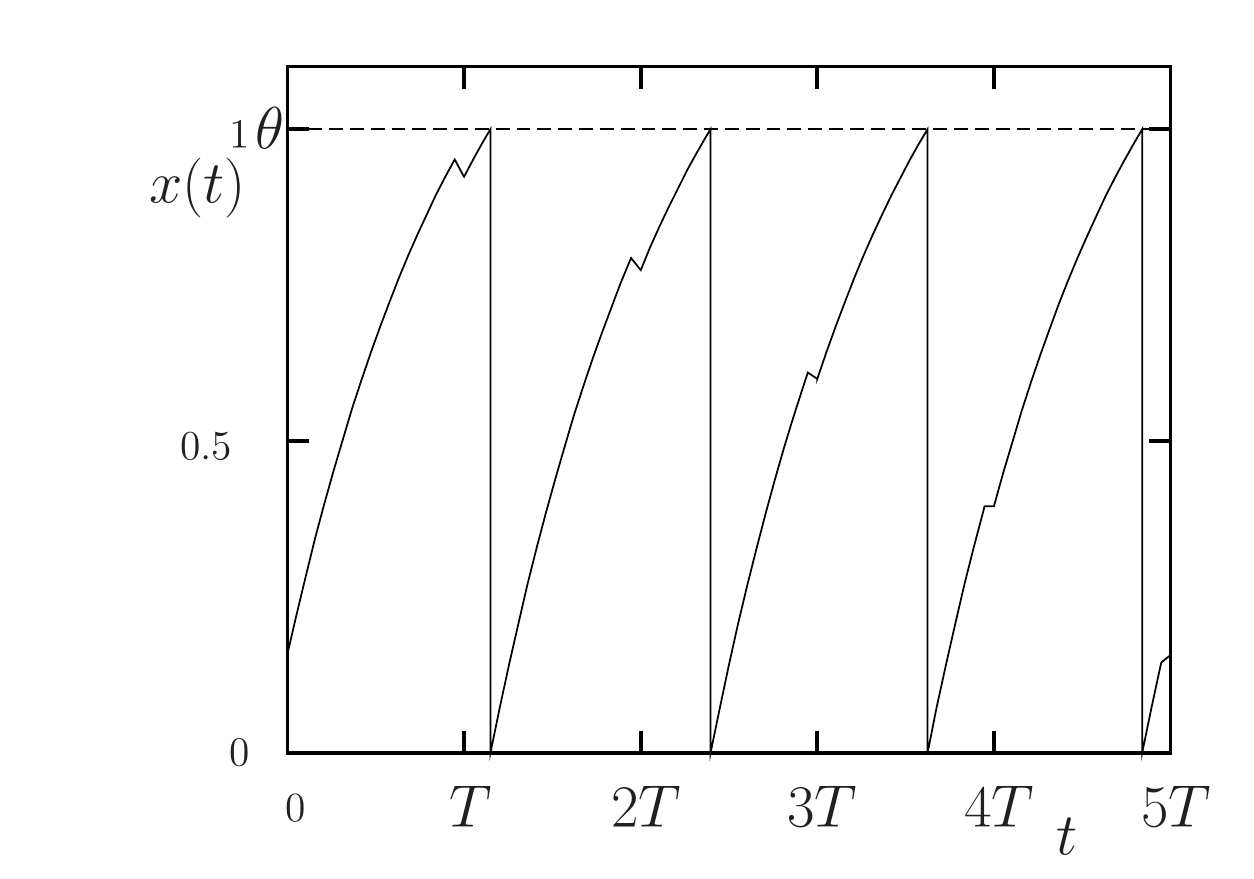}}
}
\end{picture}
\caption{Period $5$ periodic orbits with different symbolic sequences and firing
number: $\LL^4\R$ $\eta=1/5$ (a), $\LL^2\R\LL\R$ $\eta=2/5$ (b), $\LL\R\LL\R^2$
$\eta=3/5$ (c) and $\LL\R^4$ $\eta=4/5$ (d).  Parameter values for which these
periodic orbits exist are marked with points in figure~\ref{fig:d-invA_T1d9},
and correspond to large values of the duty cycle $d$.}
\label{fig:period5_orbits_larged}
\end{figure}

The firing number associated with the periodic orbits is related with their
rotation number through equation~\eqref{eq:firing-rotation_number}, and it is
shown in figure~\ref{fig:d-invA_1dscann_FN} for the periodic orbits found the
line labeled in fig.~\ref{fig:d-invA_T1d9}. As one can see there and as it comes
from proposition~\ref{pro:adding} (see \S~\ref{sec:symbolic}), the firing number
is a strictly increasing and unbounded devil's staircase as a function of the
parameters.\\

Finally, we show that the same results apply for system~\SYSTEMWR{} with $f(x)$
not necessary linear as long as conditions~\conds{} are satisfied. To illustrate
this we choose two different functions for $f$:
\begin{equation}
\begin{aligned}
f_2(x)&=a_2\left( x-b_2 \right)^5-c_2x\\
f_3(x)&=-\arctan(a_3(x-b_3)).
\end{aligned}
\label{eq:functions}
\end{equation}
With $\theta=1$, we choose parameter values in order to make the functions
$f_2$ and $f_3$ satisfy condition~\conds{}:
\begin{center}
\begin{minipage}[c]{0.2\textwidth}
\begin{align*}
a_2&=-10\\ b_2&=0.7\\c_2&=0.01
\end{align*}
\end{minipage}
\begin{minipage}[c]{0.2\textwidth}
\begin{align*}
a_3&=100\\ b_3&=0.1
\end{align*}
\end{minipage}
\end{center}
\begin{figure}
\begin{center}
\begin{picture}(1,0.4)
\put(0,0.35){
\subfigure[]{\includegraphics[angle=-90,width=0.5\textwidth]
{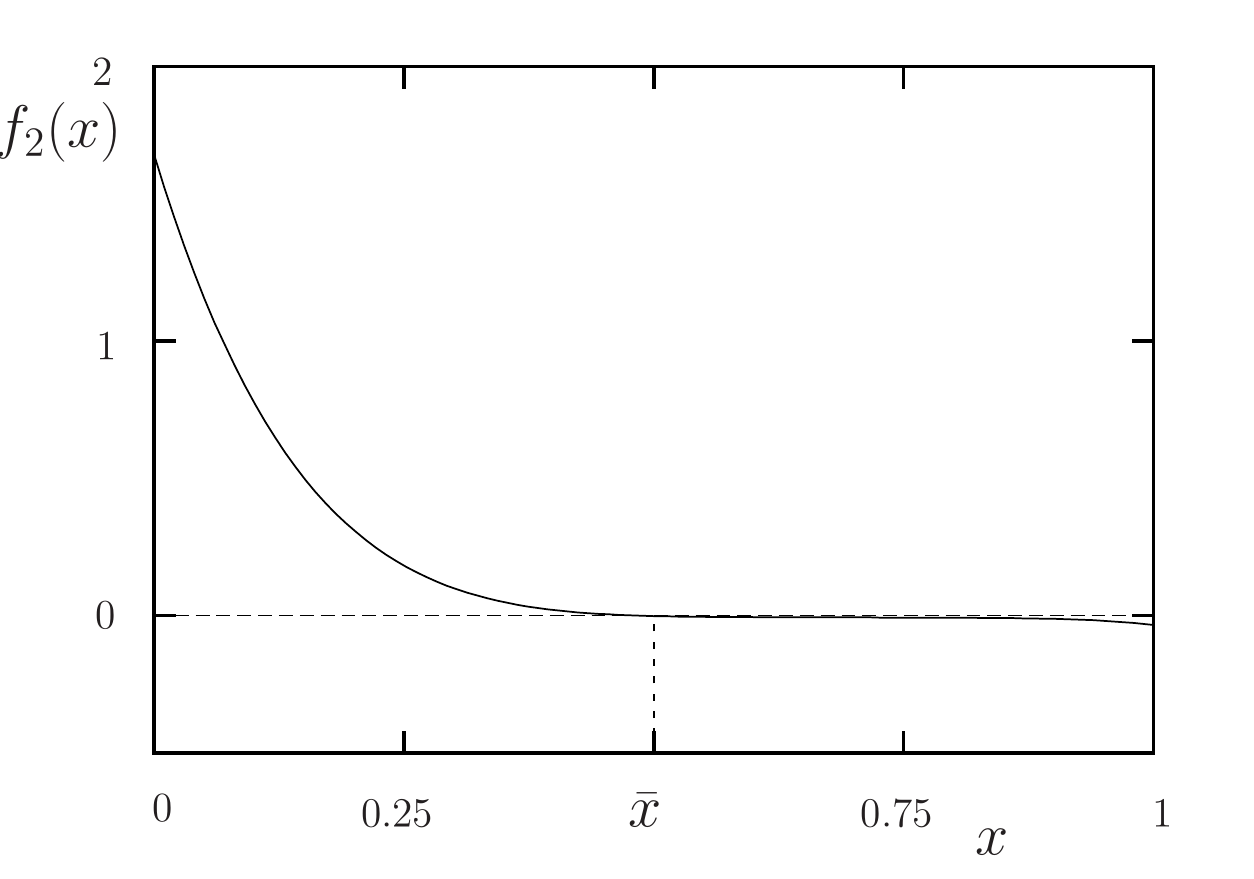}}
}
\put(0.5,0.35){
\subfigure[]{\includegraphics[angle=-90,width=0.5\textwidth]
{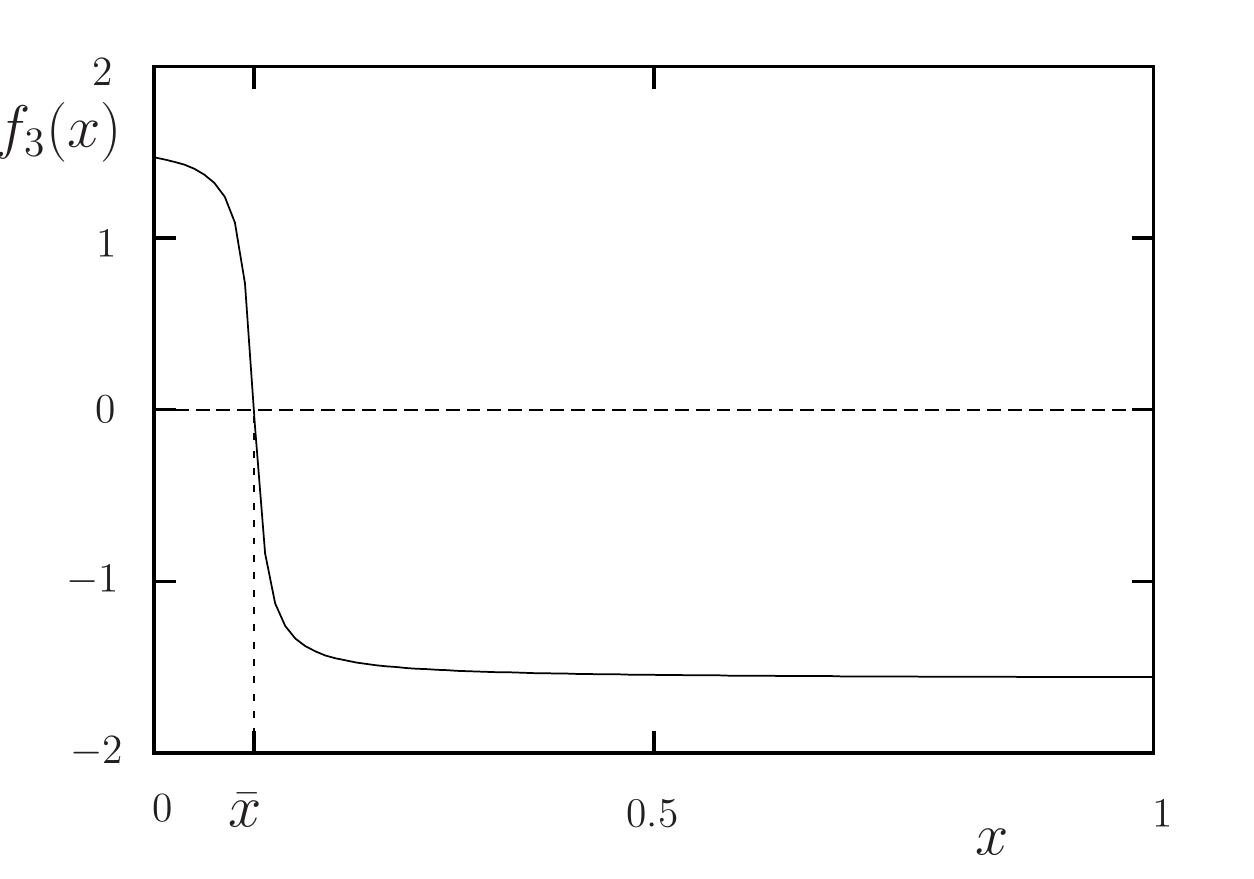}}
}
\end{picture}
\end{center}
\caption{Functions $f_2$ and $f_3$ given in eq.~\eqref{eq:functions}. $\bx$ is
the attracting equilibrium point.}
\label{fig:functions}
\end{figure}
These two functions are shown for these parameter values in
fig.~\ref{fig:functions}. There one can see that the first case corresponds to a
weak equilibrium point whereas the second one to a robust one. This leads to
slow and fast subthreshold dynamics, respectively.\\
The bifurcation structures in the parameter space $d\times 1/A$ are shown in
fig.~\ref{fig:d-invA_poly_atan}, which are as predicted. The bifurcation
scenarios along curves transversally crossing the bifurcation curves are
equivalent to the ones for the first example, as predicted by our results, and
we do not show further details.\\
The bifurcation scenario in the $d\times 1/A$ plane for the weak equilibrium
point is mainly covered by periodic orbits (bursting spiking), whereas in the
second case, the spiking dynamics is mainly given by fixed points of the
stroboscopic (tonic spiking).  This is because the subthreshold dynamics for
$f_2$ is much slower than for $f_3$. This can be seen in
fig.~\ref{fig:LRLR2_poly_atan}, where we show for both systems a $5$-periodic
orbit with symbolic sequence  $\LL\R\LL\R^2$. Note that the value of $T$ for
$f_3$ in fig.~\ref{fig:d-invA_atan} had to be reduced with respect to
fig.~\ref{fig:d-invA_poly} in order to observe bursting spiking (periodic
orbits), because, due to the fast subthreshold dynamics, for the same value of
$T$ as for $f_2$ one basically observes tonic spiking (fixed points).  When
reducing the period, the fast convergence of the subtreshold dynamics for $f_3$
is compensated and hence the regions in parameter space locating fixed points
become broauder and easier observable. We refer to~\cite{GraKruCle13b} for a
complete study of the behaviour of the bifurcation curves under frequency
variation of the input.

\begin{figure}
\begin{center}
\begin{picture}(1,0.4)
\put(0,0.35){
\subfigure[\label{fig:d-invA_poly}]{\includegraphics[angle=-90,width=0.5\textwidth]
{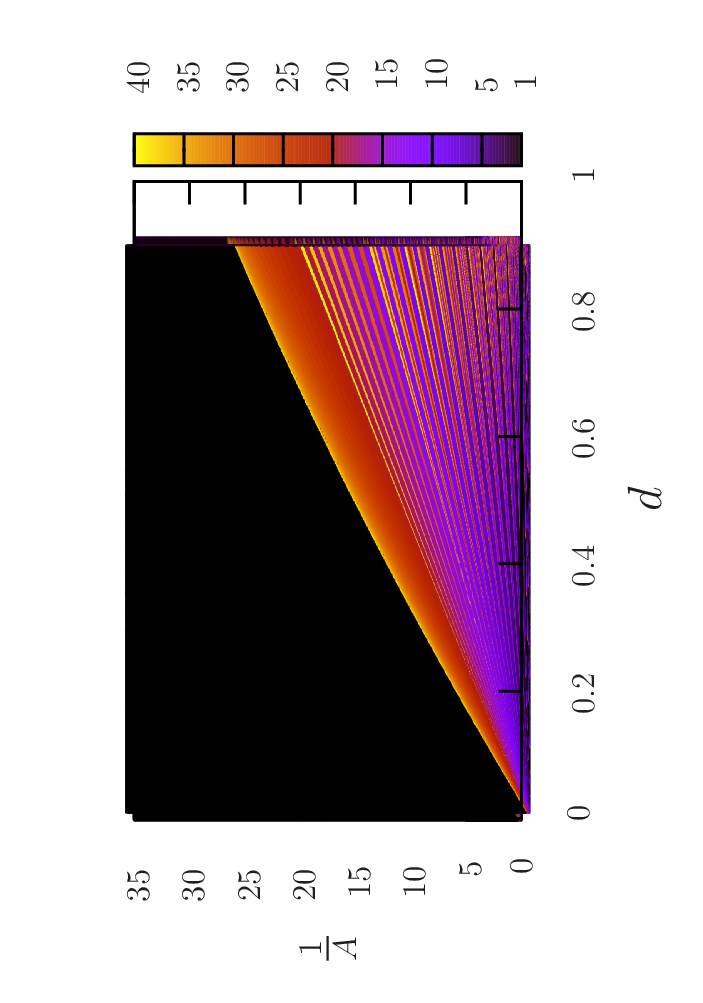}}
}
\put(0.5,0.35){
\subfigure[\label{fig:d-invA_atan}]{\includegraphics[angle=-90,width=0.5\textwidth]
{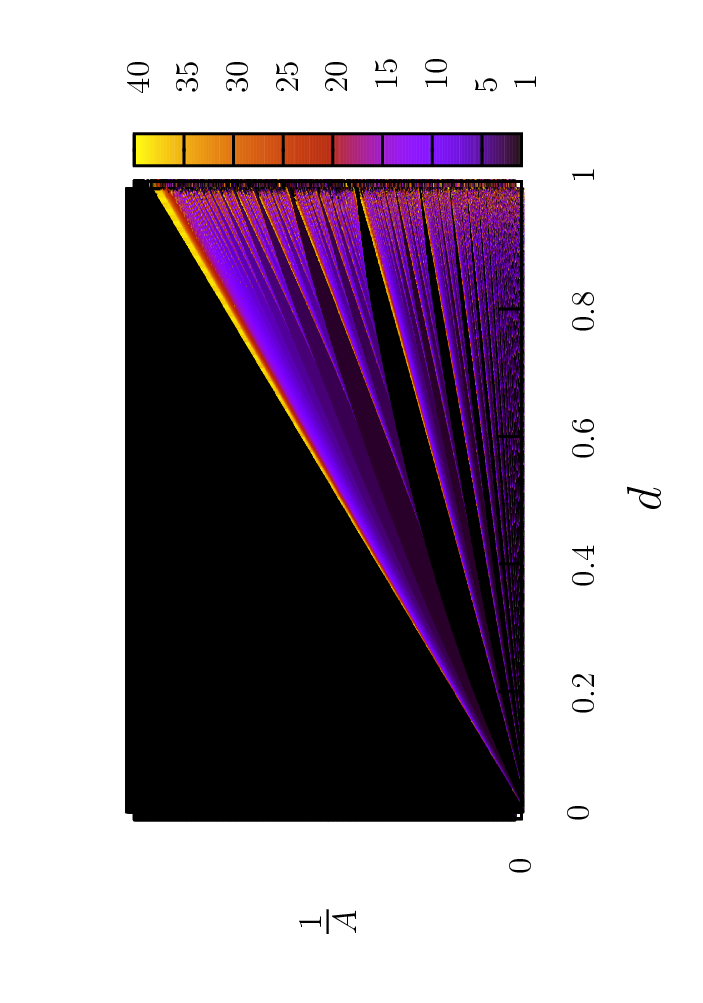}}
}
\end{picture}
\end{center}
\caption{Bifurcation scenarios for the second and third examples given in
eq.~\eqref{eq:functions}. (a) and (b) correspond to $f_2$ and $f_3$,
respectively. The values of the period $T$ have been $T=1$ for (a) and $T=0.5$
for (b).}
\label{fig:d-invA_poly_atan}
\end{figure}

\begin{figure}
\begin{center}
\begin{picture}(1,0.4)
\put(0,0.35){
\subfigure[]{\includegraphics[angle=-90,width=0.5\textwidth]
{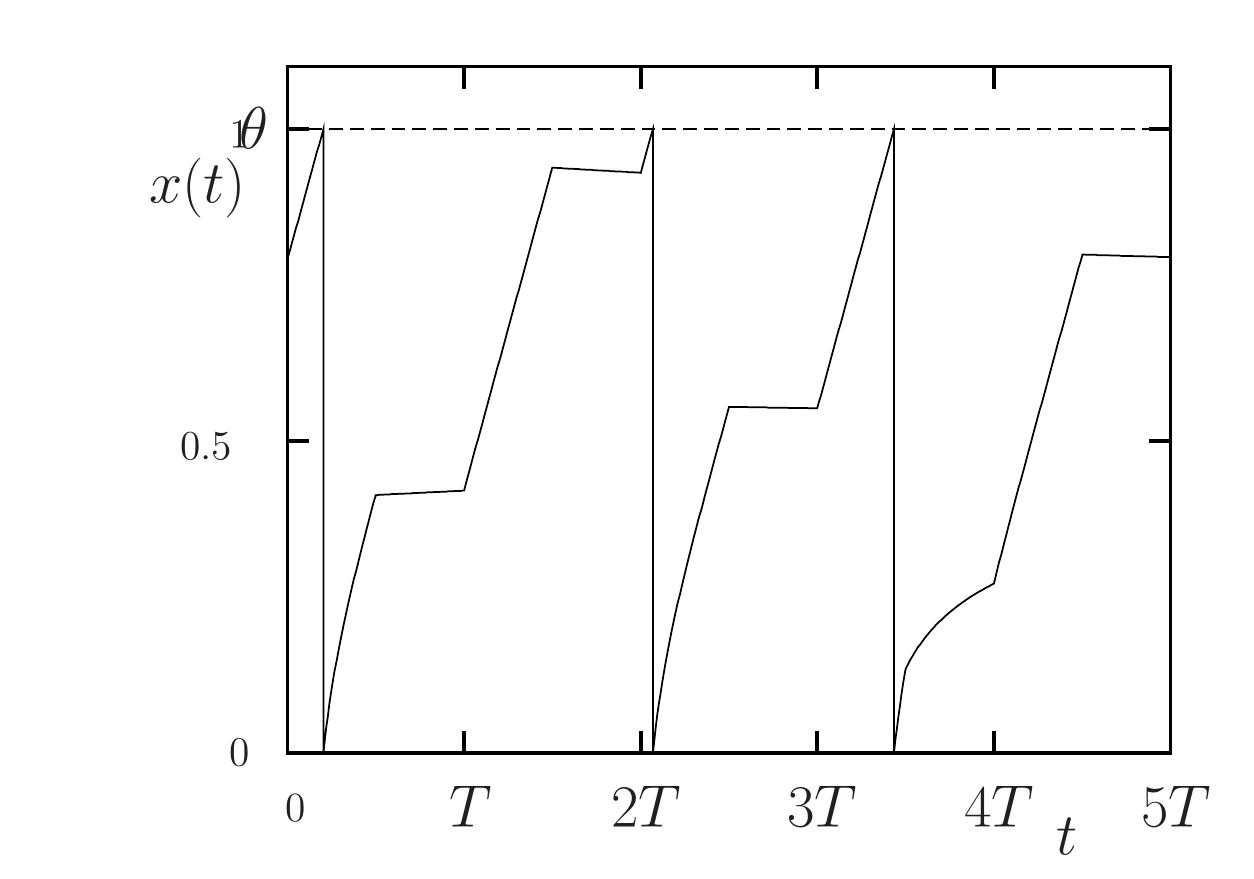}}
}
\put(0.5,0.35){
\subfigure[]{\includegraphics[angle=-90,width=0.5\textwidth]
{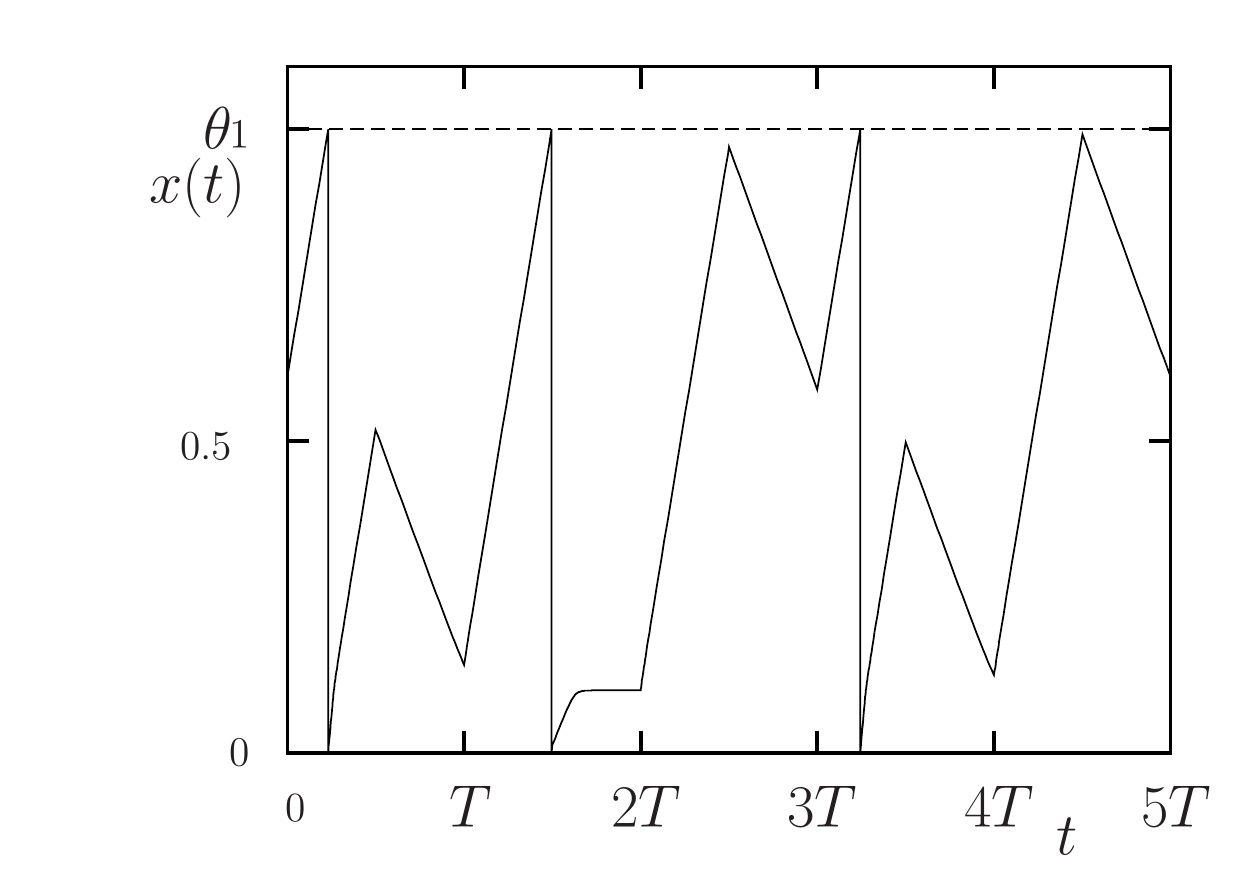}}
}
\end{picture}
\end{center}
\caption{Period-$5$ periodic orbits with symbolic sequence $\LL\R\LL\R^2$ and
firing number $\eta=3/5$ for the examples given in eq.~\eqref{eq:functions}.
(a) corresponds to  $f_2$  and (b) to $f_3$. Parameter values are $1/A=0.79$ and
$1/A=0.2$ for (a) and (b), respectively, and $d=0.5$ in both cases.}
\label{fig:LRLR2_poly_atan}
\end{figure}

\section{Discussion}
Hybrid systems with resets are a simplified version of excitable systems and are
widely used in modeling of biological systems, e.g. in the context of neuronal
activity and secretion of hormones. Typically, as such systems undergo
discontinuities (spikes), they are studied  by means of firing maps, or impact
maps, i.e. Poincar\'e maps defined using the threshold as a Poincar\'e section.  In
this work we have studied generic periodically forced hybrid systems under the
assumption of monotonicity and contracting dynamics.  The main innovation of
this paper was to use the stroboscopic (time $T$ return map) rather than the
return map.  We have obtained a complete description of the bifurcation
structure and the dynamics, showing the existence of a globally attracting
periodic orbit for every parameter value as well as giving a description of the
transitions between different regions of stable periodicity.  We have also
introduced the rotation number and shown that its typical dependence on
parameters is a devil's staircase. Finally, we defined firing rate in the
context of these systems and related it to the rotation number.\\

It is important and interesting to compare the method developed in this paper
with the approach based on the firing map. The main reason why the firing map has
been the tool of choice is that in certain contexts it becomes a regular  map;
however, the analysis relies on explicit computation of the firing times and
hence this method may not be suitable to provide general results that can be
systematically applied. Firing maps become even more difficult to use in the
presence of time dependent periodic forcings, as one needs to check for
congruency between spiking times and the period of the forcing.

The property of smoothness, which is the main advantage of the firing map, holds
automatically for systems in one dimension as the threshold is a real number
which is identified with a different real number by means of the reset.
However, in higher dimensions, (when considering dynamical adaption currents or
dynamical thresholds) the threshold becomes a codimension-one manifold and one
can find much richer dynamics. In fact, the firing map may exhibit
discontinuities near those points where the flow exhibits tangencies with the
threshold. As shown in~\cite{CooThuWed12,JimMihBroNieRub13} for two-dimensional
linear systems that can be solved explicitly, the firing map may exhibit
opposite slopes on both sides of such discontinuities. In this case, one can
apply the results for piecewise-smooth maps given in~\cite{AvrGraSch11} to
obtain a full description of the \emph{bursting} spiking dynamics (periodic
orbits of the map) and its symbolic dynamics.

The advantage of using the stroboscopic map  when considering periodically
forced systems with resets is that, after proper reparametrizations, it can be
classified as a discontinuous map with positive slopes near the discontinuity.
Such maps have been extensively studied and are well
understood~\cite{CouGamTre84,GamGleTre84,GamLanTre84,GamProThoTre86,Gam87,ProThoTre87,TurShi87,GamGleTre88,GamTre88,LyuPikZak89,AvrSch06}.
This allows us to provide general results that can be systematically applied to
systems satisfying generic conditions. Moreover, the whole analysis of the
period adding bifurcation structures comes from understanding only the
bifurcations undergone by fixed points, and is independent of the number of
spikes exhibited by the periodic orbits corresponding to these fixed points for
the time continuous system.  Finally, the presented framework is well suited for
generalizations of the considered systems. For example, a similar analysis can
be done for general periodic forcings, systems with expansive subthreshold
dynamics or even systems in higher dimensions, as could be the case of
considering a dynamical threshold.

\section*{Acknowledgements}
We would like thank Jean-Marc Gambaudo for his helpful comments on discontinuous
maps.
\newcommand{\etalchar}[1]{$^{#1}$}
\def\zh{Zh}\def\yu{Yu}\def\ya{Ya}

\end{document}